\documentclass{article}
\usepackage{graphicx} 
\usepackage{amsmath}
\usepackage{amsfonts}
\usepackage{amsthm}
\usepackage{amssymb}
\usepackage{mathrsfs}
\usepackage[margin=1.2in]{geometry}
\usepackage[shortlabels]{enumitem}
\usepackage[thinlines]{easytable}

\newcommand{\G}{\mathcal{G}}
\newcommand{\R}{\mathbb{R}}
\newcommand{\C}{\mathbb{C}}
\newcommand{\g}{\mathfrak{g}}

\newcommand{\p}{\mathfrak{p}}

\renewcommand{\sl}{\mathfrak{sl}}
\newcommand{\h}{\mathfrak{h}}

\newcommand{\tr}{\mathrm{tr}}
\newcommand{\ad}{\mathrm{ad}}

\newcommand{\gl}{\mathfrak{gl}}

\newcommand{\z}{\mathfrak{z}}
\newcommand{\s}{\mathfrak{s}}
\newcommand{\im}{\mathrm{im}}
\newcommand{\Ad}{\mathrm{Ad}}
\newcommand{\Ric}{\mathrm{Ric}}

\renewcommand{\P}{\mathbf{P}}
\newcommand{\1}{\overline{1}}
\newcommand{\Id}{\mathrm{Id}}
\newcommand{\End}{\mathrm{End}}
\newcommand{\E}{\mathcal{E}}
\newcommand{\F}{\mathcal{F}}

\renewcommand{\L}{\mathcal{L}}

\newtheorem{theorem}{Theorem}[section]
\newtheorem{proposition}[theorem]{Proposition}
\newtheorem{lemma}[theorem]{Lemma}
\newtheorem{corollary}[theorem]{Corollary}

\newtheorem{example}[theorem]{Example}
\newtheorem{remark}[theorem]{Remark}
\newtheorem*{thmrhoTensor}{Theorem \ref{thm::rhoTensor}}

\title{WEYL FORMS IN PARABOLIC CONTACT GEOMETRIES}
\author{Toby Aldape}
\date{}

\begin{document}

\maketitle

\begin{abstract} 
   In a parabolic contact geometry a choice of contact form determines a Weyl form consisting of three components: a soldering form, a Weyl connection, and a Rho tensor. An explicit description of the soldering form is well-known. We express the Rho tensor in terms of the Weyl connection in the setting of torsion-free normal parabolic contact geometries. Parabolic contact geometries with a specified contact form have a second canonical connection, the Tanaka--Webster connection. We express the Weyl connection in terms of the Tanaka--Webster connection in the setting of normal parabolic contact geometries. As an application of some of these results we derive formulas for the harmonic curvature components in integrable Legendrian contact geometry, recovering a result of S. Ivanov, D. Vassilev, and S. Zamkovoy \cite{IvanovVassilevZamkovoy}.
\end{abstract}

\section{Introduction}
\subsection{Motivation}
For parabolic geometries, or Cartan geometries modeled on flag varieties, a theory detailed in the standard reference text \cite{CapSlovak2} exhibits an equivalence between Cartan geometries and underlying data. However, translating between Cartan geometries and their underlying data is difficult in general. The Weyl structures developed in \cite{CapSlovak1} provide a bridge between these two perspectives. The Weyl form induced by a Weyl structure is closely linked to the Cartan connection, but also permits concrete tensor calculus. Cartan geometries can equivalently be understood in terms of tractor bundles. From this perspective, a choice of Weyl structure induces a splitting of the tractor bundle.

The Weyl form induced by a Weyl structure is composed of three pieces of data: the soldering form, the Weyl connection, and the Rho tensor. Understanding these three components is equivalent to understanding the tractor connection relative to the splitting of the tractor bundle induced by the Weyl structure. Knowing the tractor connection can be a useful way to express invariant differential operators. For example, Gover and Graham used the CR tractor connection to express invariant operators with leading terms equal to powers of the CR sub-Laplacian \cite{GoverGraham}. Tractor connections are also objects of interest because parallel sections of a tractor connection or a modification of a tractor connection may correspond to geometrically interesting objects. Some examples of this are conformal Killing vector fields, an Einstein metric in a conformal class, or a metric inducing a projective structure.

In this paper we apply an algorithm developed in \cite{CapSlovak2} to derive formulas for the Rho tensor of a torsion-free normal parabolic contact geometry in terms of the Weyl connection. While this quantity is already known in the best-studied geometries where the result applies, our paper provides a uniform treatment of the various parabolic contact geometries.

Once the Rho tensor is expressed in terms of the Weyl connection, it remains to understand the dependence of the Weyl connection on the underlying data. There is a second canonical connection associated to a parabolic contact geometry with a fixed contact form, the Tanaka--Webster connection. The Tanaka--Webster connection is simpler to define than the Weyl connection and appears more often in the literature on CR, Legendrian contact and contact projective structures. A simple recipe for expressing the Weyl connection in terms of the Tanaka--Webster connection is given in \cite{CapSlovak2}. We carry out this recipe on a case-by-case basis depending on the complexification of the contact-graded Lie algebra defining the parabolic geometry.

In addition to invariant operators, a second natural application of the Weyl form is to find explicit curvature formulas. The most important curvature components in normal parabolic geometries are the harmonic curvature components. These invariants are much smaller than the full curvature tensor, but still provide a complete obstruction to flatness, or local isomorphism to the model geometry. In many parabolic contact geometries the harmonic curvature agrees with the contact torsion, which is simple to write in terms of the Weyl or Tanaka--Webster connections. The main exceptions are contact projective, CR, and Legendrian contact geometries. Our formulas can be used to describe the harmonic curvature components in the first two geometries, or in Legendrian contact geometry under the assumption of integrability. As an application of our results, we derive these components for integrable Legendrian contact structures. 

\subsection{Results}
Let $(\G\to M,\omega)$ be a normal parabolic contact geometry modeled on a contact-graded parabolic model geometry $(\g,P)$. The manifold $M$ admits a canonical contact distribution $H\le TM$. Let $Q=TM/H$. Fix a contact form $\theta$ for $H$, equivalently a nonvanishing section of the bundle $Q^*$. This contact form specifies an exact Weyl structure, which itself induces a Weyl form. The Weyl form is composed of three pieces of data. These are \textbf{(1)} an isomorphism $TM\cong Q\oplus H$ called the soldering form, \textbf{(2)} an affine Weyl connection $\nabla$ on $TM$, and \textbf{(3)} a $1$-form valued in the cotangent bundle $\P\in \Omega^1(M,T^*M)$ called the Rho tensor. Let $r$ be the Reeb field associated to $\theta$. Define the \emph{Weyl torsion} $\tau:H\to H$ by $\tau(X)=\nabla_r X -[r,X]$. In this paper indices from the Latin alphabet like $i,j$ always refer to basis vectors in $H$, while the index $0$ refers to the Reeb field $r$. The partially defined $2$-form $L(X,Y)=\theta([X,Y])$ on $H$ can be used to raise and lower Latin indices. Let $n=\frac{1}{2}\dim(H)$. The following is our main result.
\begin{theorem}[Rho Tensor]\label{thm::RhoTensor}
 The Rho tensor $\P$ induced by a contact form on a torsion-free normal parabolic contact geometry is given by
\begin{itemize}
\item $\P_{ij}=-\tau_{ij}$,
\item $\P_{i0}=\frac{1}{2n+1}\nabla^j \tau_{ij}$,
\item $\P_{0i}=\frac{2}{2n+1}\nabla^j \tau_{ij}$,
\item $\P_{00}
=\frac{1}{n}\left( 
\frac{-1}{2n+1}\nabla^i \nabla^j 
\tau_{ij}
+\frac{1}{2}\tr(\tau\circ \tau)
\right)$.
\end{itemize}
\end{theorem}
\begin{remark}
    Examples of torsion-free normal parabolic contact geometries include CR geometries, integrable Legendrian contact geometries, and contact projective geometries with vanishing contact torsion, in addition to all flat parabolic contact geometries.
\end{remark}
\begin{remark}
    In flat geometries, there always exist local Weyl structures for which the Rho tensor vanishes. This can be seen to be equivalent to the vanishing of the Weyl torsion $\tau$. In the language of CR geometry, this happens exactly when the pseudohermitian structure $\theta$ is Sasakian and Ricci-flat.
\end{remark}

The choice of contact form determines a second affine connection, the Tanaka--Webster connection $\widehat{\nabla}$, which was defined in this generality in \cite{CapSlovak2}. The Reeb field $r$ is parallel relative to both $\nabla$ and $\widehat{\nabla}$, that is $\nabla r=\widehat{\nabla}r=0$. The two connections also agree in contact directions, meaning that for $X\in H$, $\nabla_X=\widehat{\nabla}_X$. It follows that if $\widehat{\nabla}$ is known then $\nabla$ is completely determined by the tensor $\rho:H\to H$ such that for $X\in H$,
\[
\nabla_r X=\widehat{\nabla}_r X-\rho(X).
\]

If $\widehat{R}_{ijk}^{\ \ \ l}$ is the curvature tensor of $\widehat{\nabla}$, there is an associated Ricci tensor defined by $\widehat{\Ric}_i^{\ j}=\frac{1}{2}\widehat{R}_{k\ i}^{\ k \ j}$. It is shown in \cite{CapSlovak2} that $\rho$ can be understood as a modified version of the Ricci tensor, and thus we call it the modified Ricci tensor. The modification depends on the Dynkin diagram type of the complexification $\g^{\C}$. If $\g^{\C}\cong C_l$ for $l\ge 2$, or if $\g^{\C}$ is an exceptional simple Lie algebra, then $\rho$ is always related to $\widehat{\Ric}$ by a constant factor. However, if $\g^{\C}\cong A_l$, for example in the case of Legendrian contact or partially integrable CR geometries, or if $\g^{\C}\cong B_l$ for $l\ge 3$ or $\g^{\C}\cong D_l$ for $l\ge 4$, for example in the case of Lie contact geometries, then the answer is more complicated.

In the case $\g^{\C}\cong A_l$, there is a neighborhood of any point in which there is a canonical unordered decomposition of $H^{\C}$ into two complex $n$-dimensional distributions. Fixing an ordering, we get $H^{\C}=\E\oplus \E$. These distributions $\E$ and $\E$ are preserved by $\widehat{\Ric}$. Define $\widehat{R}$ to be the trace of $\widehat{\Ric}$ acting on $\E$. It follows that $-\widehat{R}$ is the trace of $\widehat{\Ric}$ acting on $\E$. Define $\iota$ to be the involution acting by $+1$ on $\E$ and $-1$ on $\E$. While these quantities depend on the ordering of $\E$ and $\E$, the quantity $\widehat{R}\iota $ does not and is therefore fully canonical.

In the case where $\g^{\C}$ belongs to the $B_l$ or $D_l$ series, $\widehat{\Ric}$ canonically decomposes into two components $\widehat{\Ric}'$ and $\widehat{\Ric}''$. The following theorem characterizes the modified Ricci tensor in each case.
\begin{theorem}[Modified Ricci Tensor]\label{thm::rhoTensor}
Given a normal parabolic contact geometry with a fixed contact form, the modified Ricci tensor $\rho$ can be computed as follows.
\begin{itemize}
\item If $\g^{\C}\cong A_l$ for $l\ge 2$, $\rho=\frac{1}{n+2}\left(\widehat{\Ric}-\frac{\widehat{R}}{2(n+1)}\iota\right)$.
\item If $\g^{\C}\cong B_l$ for $l\ge 3$ or $\g^{\C}\cong D_l$ for $l\ge 4$, $\rho=\frac{\widehat{\Ric}'}{2n}+\frac{\widehat{\Ric}''}{n+4}$.
\item If $\g^{\C}\cong C_l$ for $l\ge 2$, $\rho=\frac{\widehat{\Ric}}{n+1}$.
\item If $\g^{\C}\cong E_6$, $\rho=\frac{\widehat{\Ric}}{16}$.
\item If $\g^{\C}\cong E_7$, $\rho=\frac{\widehat{\Ric}}{24}$.
\item If $\g^{\C}\cong E_8$, $\rho=\frac{\widehat{\Ric}}{40}$.
\item If $\g^{\C}\cong F_4$, $\rho=\frac{\widehat{\Ric}}{12}$.
\item If $\g^{\C}\cong G_2$, $\rho=\frac{3\widehat{\Ric}}{16}$.
\end{itemize}
\end{theorem}
Let $\widehat{\tau}:H\to H$, the Webster torsion, be defined by $\widehat{\tau}(X)=\widehat{\nabla}_rX-[r,X]$. According to the definition of the Tanaka--Webster connection, the Webster torsion is sharply constrained. More specifically, $\widehat{\tau}$ is harmonic. With Theorem~\ref{thm::rhoTensor} in place it can be helpful to express the results of Theorem~\ref{thm::RhoTensor} in terms of the decomposition
\[
\tau=\widehat{\tau}-\rho.
\]
Finally, we apply some of our results to compute the harmonic curvature invariants in integrable Legendrian contact geometry in terms of the Tanaka--Webster connection. In dimension $3$ we define tensors $Q_{11}$ and $Q_{\1\1}$ in equations \eqref{eqn::Umbilical1} and \eqref{eqn::Umbilical2}, analogs of the Cartan umbilical tensor in CR geometry. In higher dimensions we define a tensor $C_{\alpha \overline{\beta}\gamma \overline{\sigma}}$ in equation \eqref{eqn::ChernMoser}, an analog of the Chern--Moser tensor in CR geometry. These tensors fully characterize flatness in integrable Legendrian contact geometry.
\begin{theorem}[Harmonic Curvature in Legendrian Contact Geometry]\label{thm::HarmonicCurvature}
\begin{enumerate}[(a)]
\item An integrable Legendrian contact manifold of dimension $3$ is locally isomorphic to the model space $\mathrm{Flag}_{1,2}(\R^{3})$ if and only if $Q_{11}=Q_{\overline{1}\overline{1}}=0$.
\item An integrable Legendrian contact manifold of dimension $2n+1\ge 5$ is locally isomorphic to the model space $\mathrm{Flag}_{1,n+1}(\R^{n+2})$ if and only if $C_{\alpha \overline{\beta}\gamma \overline{\sigma}}=0$.
\end{enumerate}
\end{theorem}

\subsection{Previous Work}
The Weyl form or tractor connection is known in many parabolic geometries. It is relatively straightforward to describe in $|1|$-graded parabolic geometries, and \cite{CapSlovak2} does so for conformal and projective structures. Weyl forms for quaternionic contact geometry, a structure involving a specified distribution of codimension $3$, were derived by Alt \cite{Alt}.

Specializing to parabolic contact geometries, the Rho tensor has been derived previously in the best-studied cases. Of the parabolic contact geometries, CR geometry seems to be the most intensively studied. In \cite{GoverGraham} Gover and Graham found the tractor connection in CR geometry. In \cite{Matsumoto} Matsumoto gave a more detailed derivation of the tractor connection that also covered partially integrable CR structures, CR structures whose Cartan connection is permitted to have torsion. In \cite{Harrison} Harrison built a tractor connection for contact projective geometries with vanishing contact torsion and Fox in \cite{Fox} generalized this to contact projective structures with possibly nonvanishing contact torsion. Our formulas for the Rho tensor only hold when the Cartan geometry is assumed to be torsion-free. We note that these results do apply to the locally symmetric parabolic contact geometries defined by Zalabova \cite{Zalabova}, since these are automatically torsion-free.

A generalized path geometry is a kind of parabolic geometry that abstracts some features of second order ODEs into a geometric structure. In dimension $3$, the notion of a generalized path geometry coincides with that of a Legendrian contact geometry. As a result, a $3$-dimensional generalized path geometry is parabolic contact, and our results on the Rho tensor and the Weyl connection apply in this case. Meanwhile, the paper \cite{CapGuo} by \v Cap and Guo has explicit formulas for the Rho tensor in generalized path geometries, which, in particular, cover the $3$-dimensional case. However, the Rho tensors found in \cite{CapGuo} correspond to a different class of exact Weyl structures than the ones used in this paper. An exact Weyl structure is induced by a choice of section of a line bundle called a bundle of scales. While in most parabolic contact geometries there is essentially only one choice for the bundle of scales, the parabolic contact geometries corresponding to the $A_l$ series are unique in the sense that there is actually a $1$-parameter family of essentially distinct bundles of scales. While our paper takes the bundle of scales to be $Q^*$, the paper \cite{CapGuo} takes the bundle of scales to be the bundle $E$, see the definition of Legendrian contact structures later in this paper. 

The Tanaka--Webster connection for CR manifolds was independently introduced by Tanaka \cite{Tanaka2} and Webster \cite{Webster}. Fox in \cite{Fox} introduced a connection associated to a contact projective structure with a fixed contact form. In \cite{CapSlovak2} \v Cap and Slov\'ak generalized both of those connections by defining a Tanaka--Webster connection in the broader context of parabolic contact geometries.

Harmonic curvature formulas for Legendrian contact structures have a long history, and we review it here. Following the literature is somewhat complicated by the fact that Legendrian contact structures go by many different names, including Lagrangian contact structures, bi-Legendrian foliations, and para-CR structures. 

As stated above, a Legendrian contact structure in dimension $3$ is exactly the same as a generalized path geometry. The two components of harmonic curvature $Q_{11}$ and $Q_{\1\1}$ were already known in some form to Cartan. Cartan studied second-order ODEs in two dimensions of the form 
\begin{equation} \label{eqn::ODE}
 y''=f(x,y,p)   
\end{equation} 
where $p=\frac{dy}{dx}$. Given such an equation, the projectivized tangent bundle of $\R^2$ inherits the structure of a generalized path geometry. From this starting point, Cartan expressed one of the two components of harmonic curvature as $f_{pppp}$ and the other as a more complicated formula involving $f$. In fact, the term $f_{pppp}$ is exactly the obstruction to solutions of \eqref{eqn::ODE} being the geodesics of some affine connection. Later, \cite{BryantGriffithsHsu} gave a more modern treatment of these two curvature components. Finally, Ivanov, Vassilev and Zamkovoy expressed the fundamental curvature invariant in Legendrian contact geometry in terms of the Tanaka--Webster connection in \cite{IvanovVassilevZamkovoy}, combining these two harmonic components into a single tensor.

In the closely analogous realm of $3$-dimensional CR geometry, the two harmonic curvature components become one, known as the Cartan umbilical tensor. This tensor was also known in some form to Cartan. However, the Tanaka--Webster connection was invented much later. The Cartan umbilical tensor was expressed in terms of the Tanaka--Webster connection in \cite{ChengLee}. Our tensors $Q_{11}$ and $Q_{\1\1}$ are nearly identical to this expression.

In dimensions higher than $3$, the fundamental curvature invariant in CR geometry is known as the Chern--Moser tensor. It was discovered by Tanaka \cite{Tanaka1} and further elaborated by Chern and Moser \cite{ChernMoser}. In \cite{Webster} Webster expressed the Chern--Moser tensor in terms of the Tanaka--Webster connection. 

In higher-dimensional Legendrian contact geometry, Montano proved in \cite{Montano} that vanishing Tanaka--Webster curvature implies local isomorphism to the model geometry. Working by analogy with the CR case, Ivanov, Vassilev and Zamkovoy found the expression for the fundamental curvature invariant in \cite{IvanovVassilevZamkovoy}, equivalent to our tensor $C_{\alpha\overline{\beta}\gamma\overline{\sigma}}$, improving Montano's result.
\subsection{Road Map}
Section 2 introduces important background used in the proof of the main result, Theorem~\ref{thm::RhoTensor}. Section 3.1 derives universal formulas for certain components of the Lie bracket of a contact-graded Lie algebra. Section 3.2 discusses the use of index notation in this paper.

Sections 4.1--4.2 introduces the Weyl torsion tensor $\tau$ and prove a contracted Bianchi identity relating it to the Weyl curvature tensor $R$. Sections 4.3--4.5 solve for the Rho tensor $\P$ homogeneity by homogeneity, concluding the proof of Theorem~\ref{thm::RhoTensor}. These sections rely on the eigenvalues of $\square$ computed in Section~\ref{app:1Cochains}.

Section 5 is concerned with the Tanaka--Webster connection and the proof of Theorem~\ref{thm::rhoTensor}, which depends on the eigenvalues of $\square$ acting on $\g_0$ proven in Section~\ref{app::0Cochains}. 

Section 6 defines Legendrian contact structures, develops some of the matrix algebra relevant to these structures, and gives an elementary definition of the Tanaka--Webster connection in this context. Finally, we use Theorem~\ref{thm::RhoTensor}, Theorem~\ref{thm::rhoTensor} and some of the components of the Cotton--York tensor $Y$ derived while proving Theorem~\ref{thm::RhoTensor} to prove Theorem~\ref{thm::HarmonicCurvature}.
\section*{Acknowledgements}
I want to thank Robert Bryant, Andreas \v Cap, Marc Herzlich and Jacob Erickson for responding either to my emails or MSE/MO questions. Part of this research was funded by the NSF Groups and Dynamics RTG (1937215) award to UT Austin.
\section{Preliminaries}
\subsection{Gradings}
Let $\g$ be a semisimple Lie algebra. A $|k|$-grading of $\g$ is a decomposition
\[
\g=\g_{-k}\oplus \cdots \oplus \g_k
\]
respecting the Lie bracket for which $\g_{-1}$ generates the negatively graded components and $\g_{-k},\g_k\neq 0$. Define $\g^i=\g_i\oplus \cdots \oplus \g_k$. Let $\g_-$ be the direct sum of the negatively graded components, $\p=\g^0$ and $\p_+=\g^1$. In particular, $\p$ is a parabolic subalgebra and $\p_+$ is its nilradical. A \emph{contact grading} is a $|2|$-grading for which $\dim(\g_{-2})=1$, the \emph{Levi bracket} $\L:\g_{-1}\times \g_{-1}\to \g_{-2}$ defined by restricting the Lie bracket is nondegenerate, and $\g_0$ contains no simple ideals of $\g$. If $\g$ is a contact-graded real Lie algebra, the complexification $\g^{\C}$ is a contact-graded complex Lie algebra. The complex contact gradings are completely classified by the fact that every complex simple Lie algebra not isomorphic to $\sl_2$ admits a unique contact grading, up to isomorphism. A real form of a complex simple Lie algebra may have at most one contact grading up to isomorphism.

The Levi subalgebra $\g_0$ has semisimple part $\g_0^{ss}=[\g_0,\g_0]$. If $\z(\g_0)$ is the center of $\g_0$, then $\g_0=\z(\g_0)\oplus \g_0^{ss}$. Let $\g$ be a complex semisimple Lie algebra with Cartan subalgebra $\h$ and positive root system $\Delta^+$. Let $\Delta^0\subset \Delta^+$ be the simple roots. As detailed in \cite{CapSlovak2}, the $|k|$-gradings of $\g$ are determined up to isomorphism by a choice of subset $I\subset \Delta^0$, where the homogeneity of a root space in $\g$ is determined by the $I$-height of its corresponding root. For a $|k|$-grading induced in this fashion, $\h=\z(\g_0)\oplus \h'$ for $\h'$ a Cartan subalgebra of $\g_0^{ss}$. 

Denote the highest root of $\g$ by $\mu$, and the $i$-th simple root by $\alpha_i$ according to the ordering of simple roots used in the diagrams in the Appendix. Denote the $i$-th fundamental weight by $\lambda_i$. The contact gradings are exactly the ones determined by the choice $I=\{\alpha\in \Delta^0: \langle \alpha, \mu \rangle\neq 0 \}$.

Every $|k|$-grading is associated to an element $E\in \z(\g_0)$ such that $[E,X]=iX$ for $X\in \g_i$, called the \emph{grading element}. We assume a fixed choice of nondegenerate, $\g$-invariant bilinear form $\langle \cdot, \cdot \rangle$ on $\g$. For $\g$ simple, there is in fact a unique choice of $\g$-invariant bilinear form up to scale. This form induces isomorphisms $\g_i\cong \g_{-i}^*$ and $\p_+\cong \g_-^*$. 

A choice of such a form is often standardized as the Killing form, but in this paper we will instead assume the normalization condition $|E|^2=2$. Let $H_{\mu}\in \h$ be the element dual to $\mu\in \h^*$ relative to our fixed bilinear form. If $\g$ is contact-graded, it is established in \cite{CapSlovak2} that $E=\frac{2}{|\mu|^2}H_{\mu}$. It follows from our normalization condition that $2=|E|^2=\frac{4}{|\mu|^2}$, so $|\mu|^2=2$ and $E=H_{\mu}$.

\subsection{Differentials}
We fix the convention on wedge product such that for a $k$-form $\omega$ and an $l$-form $\eta$, 
\[
\omega\wedge \eta(X_1,\dots,X_{k+l})=\frac{1}{k!l!}\sum_{\sigma}\mathrm{sgn}(\sigma)\omega(X_{\sigma(1)},\dots,X_{\sigma(k)})\eta(X_{\sigma(k+1)},\dots,X_{\sigma(k+l)}).
\]
Let $\partial^*:C^k(\g/\p,\g)\to C^{k-1}(\g/\p,\g)$ be the Kostant codifferential. This depends on both the fixed bilinear form and the wedge convention. We have the following formula to compute the Kostant codifferential of a $2$-cochain. Let $X_i$ be a basis for $\g_-$ and let $X^i$ be a dual basis in $\p_+$. If $\phi\in C^2(\g/\p,\g)$ is a $2$-cochain then for $X\in \g/\p$, 
\begin{equation}\label{eqn::KostantCod1}
\partial^*\phi(X)=\sum[X^i,\phi(X,X_i)]-\frac{1}{2}\sum \phi([X^i,X],X_i).
\end{equation}
Taking advantage of the isomorphism $C^2(\g/\p,\g)\cong C_2(\p_+,\g)$, the codifferential acts on a decomposable element by
\begin{equation}\label{eqn::KostantCod2}
\partial^*(X^1\wedge X^2\otimes A)=-X^2\otimes [X^1,A]+X^1\otimes [X^2,A]-[X^1,X^2]\otimes A.
\end{equation} 
It is helpful to compare these two expressions for the Kostant codifferential. The first term of \eqref{eqn::KostantCod1} corresponds to the first two terms of \eqref{eqn::KostantCod2}, and the second term of \eqref{eqn::KostantCod1} corresponds to the third term of \eqref{eqn::KostantCod2}.

By identifying $\g/\p\cong \g_{-}$, we can also think of the Kostant codifferential as a map $\partial^*:C^k(\g_-,\g)\to C^{k-1}(\g_-,\g)$.
Let $\partial:C^k(\g_-,\g)\to C^{k+1}(\g_-,\g)$ be the differential. On a $0$-cochain $\phi$ this acts by
\[
\partial \phi(X)=[X,\phi]
\]
and on a $1$-cochain $\partial$ acts by
\[
\partial \phi(X,Y)=[X,\phi(Y)]-[Y,\phi(X)]-\phi([X,Y]).
\]
The \emph{Kostant Laplacian} $\square:C^k(\g_-,\g)\to C^k(\g_-,\g)$ is defined by $\square=\partial \partial^*+\partial^* \partial$. 

\subsection{Cartan Geometry}
A \emph{model geometry} is a pair $(\g,P)$ consisting of a Lie algebra $\g$, a Lie group $P$ such that $\p=\mathrm{Lie}(P)$ is a subalgebra of $\g$, and an action $P\curvearrowright \g$ by automorphisms that extends the adjoint action by $P$ on $\p$ and induces the adjoint action by $\p$ on $\g$. This is most often specified by a pair $(G,P)$ of a Lie group $G$ and a closed subgroup $P$, and then $\g$ is defined to be the Lie algebra of $G$. A \emph{parabolic model geometry} is a model geometry $(\g,P)$ where $\g$ is semisimple with a fixed $|k|$-grading and $\p=\mathrm{Lie}(P)=\g_0\oplus \cdots \oplus \g_k$. A \emph{parabolic geometry} is a Cartan geometry modeled on a parabolic model geometry $(\g,P)$.

Let $(\G\to M,\omega)$ be a Cartan geometry modeled on the parabolic model geometry $(\g,P)$. The \emph{curvature} of this Cartan geometry is a $\g$-valued $2$-form $\Omega=d\omega+\frac{1}{2}[\omega,\omega]$. Because $\Omega$ is horizontal and $P$-equivariant, it corresponds to a section of $\G\times_P C^2(\g/\p,\g)$. Since $\partial^*:C^2(\g/\p,\g)\to C^1(\g/\p,\g)$ is $P$-equivariant, it induces a map $\partial^*:\G\times_PC^2(\g/\p,\g)\to \G\times_P C^1(\g/\p,\g)$. The Cartan geometry is \emph{normal} if $\partial^* \Omega=0$. The Cartan geometry is \emph{regular} if $\Omega_p\in C^2(\g/\p,\g)$ is contained in positive homogeneity for all $p\in \G$. There are some parabolic geometries of interest that are not normal, such as contact projective structures with nonvanishing contact torsion. However, regularity and normality are fairly standard and we will assume them throughout. The Cartan geometry is \emph{torsion-free} if $\Omega_p(X,Y)\in \p$ for all $p\in \G$ and $X,Y\in \g/\p$. Notice that torsion-freeness implies regularity.

A \emph{parabolic contact geometry} is a regular parabolic geometry modeled on a contact-graded pair $(\g,P)$. If $(\pi:\G\to M,\omega)$ is a parabolic contact geometry, define a codimension-one distribution $H=\pi_*(\omega^{-1}(\g^{-1}))$. It is a consequence of regularity that $H$ is a contact distribution. 

\subsection{Weyl Structures}
Fix a parabolic contact structure $(\pi:\G\to M,\omega)$ modeled on $(\g,P)$. Let $G_0=N_P(\g_0)$ and let $P_+=\exp(\p_+)\le P$. If $(\G\to M,\omega)$ is a parabolic geometry modeled on $(\g,P)$ let $\G_0=\G/P_+$, a $G_0$ principal bundle. A \emph{Weyl structure} is a $G_0$-equivariant section $\sigma:\G_0\to \G$ of the projection $\G\to \G_0$. The pullback $\sigma^* \omega$ is the \emph{Weyl form} associated to $\sigma$. Let $\sigma^*\omega=\omega_-+\omega_0+\omega_+$ be a decomposition into $\g_-,\g_0$, and $\p_+$-valued components. Then $\omega_-$ is the \emph{soldering form}, $\omega_0$ is the \emph{Weyl connection}, and $\omega_+$ is the \emph{Rho tensor}.

The Cartan connection $\omega$ on $\G$ induces isomorphisms $\G_0\times_{G_0}\g_{-2}\cong Q$ and $\G_0\times_{G_0}\g_{-1}\cong H$. Similarly, there are isomorphisms $\G_0\times_{G_0} \g_1\cong H^*$ and $\G_0\times_{G_0}\g_2 \cong Q^*$. The soldering form $\omega_-$ is horizontal and $G_0$-equivariant, so it corresponds to a $1$-form valued in the associated bundle $\G_0\times_{G_0}\g_-\cong Q\oplus H$, which induces an isomorphism $TM\cong Q \oplus H$. The form $\omega_0$ is a principal connection on $\G_0$, which induces vector bundle connections on associated bundles. In particular, we get a connection on $Q\oplus H$ which, when coupled with the soldering form, induces an affine connection $\nabla$ on $TM$ which we will also call the Weyl connection. The Rho tensor $\omega_+$ is horizontal and $G_0$-equivariant, so it corresponds to a $1$-form valued in the associated bundle $\G_0\times_{G_0}\p_+\cong H^*\oplus Q^*$. On the other hand, the soldering form provides an isomorphism $H^*\oplus Q^*\cong TM^*$, so the Rho tensor induces a $1$-form valued in the cotangent bundle $\P\in \Omega^1(M,T^*M)$, equivalently a $(0,2)$-tensor.

A \emph{scaling element} is an element $F\in \z(\g_0)$ such that $\ad(F)$ acts injectively on each $\g_i$ for $i\neq 0$. Suppose $\lambda:G_0\to \R^{\times}$ is a Lie group homomorphism such that $\lambda_*(X)=\langle F, X \rangle$ for $F$ a scaling element. Then the associated bundle $\mathcal{L}^{\lambda}=\G_0\times_{G_0} \R$ is called a bundle of scales. Given a Weyl structure $\sigma:\G_0\to \G$, the Weyl connection $\omega_0$ induces a connection on the line bundle $\L^{\lambda}$. This actually defines a one-to-one correspondence between Weyl structures $\sigma$ and connections on $\L^{\lambda}$ \cite{CapSlovak1}. In particular, nonvanishing global sections of $\L^{\lambda}$ induce flat connections which correspond to a subclass of Weyl structures called \emph{exact Weyl structures}.

Specializing to the case of parabolic contact geometries, consider $\lambda:G_0\to \R^{\times}$ induced by the adjoint action of $G_0$ on $\g_2$. Then $\G_0\times_{G_0} \g_2\cong Q^*$ is the bundle of contact forms for the contact distribution $H$. For $H\in \z(\g_0)$ and $X\in \g_2$, 
\[
\Ad_*(H)(X)
=\mu(H)X
=\langle H_{\mu}, H \rangle X
=\langle E, H \rangle X.
\]
Since $E$ is a scaling element, $Q^*$ is a bundle of scales. Therefore a contact form $\theta$, a global section of $Q^*$, induces a Weyl structure and thus a Weyl form. By definition, the Weyl connection induced by $\theta$ has the property that $\nabla \theta=0$ when $\theta$ is considered as a section of $Q^*$.

Let $r$ be the Reeb field. Let $\pi:TM\to H$ be the projection with kernel generated by the Reeb field. It is established in \cite{CapSlovak2} that the soldering form associated to $\theta$ induces the isomorphism $TM\cong Q\oplus H$ given by $X\mapsto (\overline{X},\pi(X))$. Therefore we may think of $Q$ as a subbundle of $TM$ by identifying it with the span of the Reeb field $r$. On the other hand, since $\theta$ is parallel relative to the connection induced on $Q^*$, it follows that $r$ is parallel relative to the connection induced on $Q$. Then we get $\nabla r=0$.

The following is Lemma 3.8(4) in \cite{CapSlovak1}.
\begin{proposition}
Let $\sigma:\G_0\to \G$ be a Weyl structure for a parabolic geometry and let $\L^{\lambda}$ be a bundle of scales. The composition of $\lambda_*$ with the curvature form of $\sigma^* \omega_0$ descends to the curvature of the induced connection on $\L^{\lambda}$.
\end{proposition}
It follows that if $\sigma$ is an exact Weyl structure, the curvature of the associated Weyl connection $\sigma^*\omega_0$ is valued in $\ker \lambda_*\subset \g_0$. Specializing to parabolic contact structures and using $\g_{-2}\cong \g_2^*$, we get the following corollary.

\begin{corollary}\label{cor::ExactWeyl}
Fix a parabolic contact geometry with the bundle of scales $Q^*$ as above. If $\sigma$ is an exact Weyl structure, the curvature $R$ of $\sigma^*\omega_0$ has output in $\g_0$ acting trivially on $\g_{-2}$.
\end{corollary}

\subsection{Levi Bracket}
Fix a parabolic contact geometry. The Levi bracket $\L\in \bigwedge^2 \g_{-1}^*\otimes \g_{-2}$ is $G_0$-invariant, so it determines a section of $\G_0\times_{G_0} (\bigwedge^2 \g_{-1}^* \otimes \g_{-2})\cong \bigwedge^2 H^* \otimes Q$. It is a consequence of regularity that the section $\L$ of $\bigwedge^2 H^* \otimes Q$ is given by $\L(X,Y)=\overline{[X,Y]}$. Let $\nabla$ be some connection on $Q\oplus H$ associated to a principal connection $\gamma$ on $\G_0$, for example the Weyl connection. Since $\L\in \bigwedge^2 \g_{-1}^*\otimes \g_{-2}$ is $\g_0$-invariant, the section $\L\in \bigwedge^2 H^* \otimes Q$ is preserved by the connection associated to $\gamma$, which is the one induced by $\nabla$. We have proven the following.
\begin{proposition}\label{prop::PreservesLevi}
    If $\nabla$ is a connection on $Q\oplus H$ induced by a principal connection on $\G_0$, then $\nabla_Z\L=0$.
\end{proposition}

\subsection{Curvature of Weyl Forms}
Since the Lie bracket in $\g$ is $G_0$-equivariant, it induces a Lie bracket on sections of $\G_0\times_{G_0}\g$, called the \emph{algebraic bracket}. The algebraic bracket of two sections $U$ and $V$ is written $\{U,V\}$.
In the presence of a soldering form, vector fields correspond to sections of $\G_0\times_{G_0}\g_-$. Thus we can take the algebraic bracket of two vector fields. As another example, define $\End_0(H)=\G_0\times_{G_0}\g_0$. Since $H\cong \G_0\times_{G_0}\g_{-1}$ and $Q\cong \G_0\times_{G_0}\g_{-2}$, the action of $\g_0$ on $\g_{-1}$ and $\g_{-2}$ induces an action by sections of $\End_0(H)$ on sections of $H$ and sections of $Q$.

Taking just the soldering form and Weyl connection, $\omega_-+\omega_0$ is a Cartan connection on $\G_0$. Let $T$ and $R$ be the torsion and curvature parts of this connection. These are both horizontal and $G_0$-equivariant, so they can be thought of as $2$-forms valued in $\G_0\times_{G_0}\g_-$ and $\G_0\times_{G_0}\g_0$, respectively. While the curvature $R$ agrees with the curvature of the Weyl connection $\nabla$, the torsion $T$ is not equal to the standard affine connection torsion. The difference is an algebraic bracket term. That is,
\[
T(X,Y)=\nabla_X Y-\nabla_Y X -[X,Y]+\{X,Y\},
\]
meaning that it agrees with the torsion of $\nabla$ when the parabolic model geometry is $|1|$-graded. Given the Rho tensor $\P\in \Omega^1(M,T^*M)$, the \emph{Cotton--York tensor} $Y\in \Omega^2(M,T^*M)$ is defined by
\begin{equation}\label{eqn::CottonYork}
Y(X,Y)=d^{\nabla}\P(X,Y)+\P(\{X,Y\})+\{\P(X),\P(Y)\}
\end{equation}
where $d^{\nabla}$ is the covariant exterior derivative induced by $\nabla$. For purposes of computation, this can be simplified a bit by the following.
\begin{proposition}\label{prop::CottonYork}
$Y(X,Y)=\nabla_X \P (Y)-\nabla_Y \P(X)+\P(T(X,Y))
+\{\P(X),\P(Y)\}$.
\end{proposition}
\begin{proof}
Summing just the first two terms of \eqref{eqn::CottonYork},
    \begin{align*}
    d^{\nabla}\P(X,Y)+\P(\{X,Y\})
    &=\nabla_X(\P(Y))
    -\nabla_Y(\P(X))-\P([X,Y])
    +\P(\{X,Y\})\\
    &=\nabla_X\P(Y)
    -\nabla_Y\P(X)
    +\P(\nabla_X Y-\nabla_Y X-[X,Y]+\{X,Y\})\\
    &=\nabla_X\P(Y)-\nabla_Y\P(X)+\P(T(X,Y)).
    \end{align*}
\end{proof}
Following \cite{CapSlovak2}, the curvature $\kappa$ of $\sigma^*\omega$ can be expressed in terms of these tensors.
\begin{equation}\label{eqn::WeylCurvature}
\kappa
=T+R+Y+\partial \P.
\end{equation}
In particular, $R,Y$, and $\P$ may only be nonzero in homogeneity at least $2$, so $\kappa^{(i)}=T^{(i)}$ for $i\le 1$. If $\kappa$ is regular, this implies $T^{(i)}=0$ for $i\le 0$, while if $\kappa$ is torsion-free then $T^{(1)}=\kappa^{(1)}=0$. As a special case of the approach outlined in \cite{CapSlovak2}, normality of $\kappa$ allows us to compute the Rho tensor $\P$ in the following way.
\begin{proposition}\label{prop::RhoTensor}
Let $M$ be a normal parabolic contact structure with a Weyl form induced by a Weyl structure. Then for $i=3,4$,
\[
\P^{(i)}=-\square^{-1}\partial^*(T+R+Y)^{(i)}.
\]
\end{proposition}
\begin{proof}
For homogeneity reasons, $\partial^*\P^{(i)}=0$. Using equation \eqref{eqn::WeylCurvature} and the fact that $\kappa$ is normal,
\begin{align*}
    0&=\partial^*\kappa^{(i)}\\
    &=\partial^*(T+R+Y+\partial \P)^{(i)}\\
    &=\partial^*(T+R+Y)^{(i)}+\square \P^{(i)}-\partial \partial^* \P^{(i)}\\
    &=\partial^*(T+R+Y)^{(i)}+\square \P^{(i)}.
\end{align*}
Since $\g$ does not contain a simple ideal isomorphic to $\sl_2(\R)$ or $\sl_2$, $H^1(\g_-,\g)^2=0$, where the superscript $2$ denotes the components in homogeneity at least $2$. Since $H^1(\g_-,\g)^2$ is isomorphic to $(\ker \square)^2\subset C^1(\g_-,\g)$, the Kostant Laplacian $\square$ is invertible in homogeneity at least $2$ and the result follows.
\end{proof}

\begin{remark}
The result is still true for a Weyl form induced by a contact form in the $i=2$ case because it is still true that $\partial^* \P^{(2)}=0$. However, proving this is a little more subtle and we do not need the result here.
\end{remark}

Note that if $\g$ is contact-graded then the torsion $T$ may be nonzero in homogeneity $1$ or $2$, the curvature $R$ may be nonzero in homogeneity $2$ or $3$, the Cotton--York tensor $Y$ may be nonzero in homogeneity $3,4$ or $5$, and the Rho tensor $\P$ may be nonzero in homogeneity $2,3$ or $4$.

The operations $\partial$ and $\partial^*$ preserve homogeneity. Therefore we can use Proposition~\ref{prop::RhoTensor} to compute the Rho tensor $\P$ by proceeding one homogeneity at a time. The basic idea is that $\nabla$ is assumed to be fixed, so $T$ and $R$ are known. Knowing $\P$ in a given homogeneity $i$ determines the Cotton--York tensor $Y$ in homogeneity $i+1$, which determines $\P$ in homogeneity $i+1$.

\section{Contact-Graded Lie Algebras}
\subsection{Lie Bracket}
There are enough structural similarities between different contact-graded Lie algebras that it is possible to get a fairly uniform understanding of the Lie bracket. The one major exception appears to be the bracket $\g_{-1}\times \g_1\to \g_0$, which seems to depend on the ``internal'' structure of $\g$. The identities proven in this section are useful in the derivation of the Rho tensor $\P$ in Section 4.

The action of an element $H\in \g_0$ on $\g_{-2}$ or $\g_2$ depends only on the inner product $\langle H,E \rangle$ with the grading element, although we do not use this fact. We identify the Levi subalgebra $\g_0$ with endomorphisms of $\g_{-1}$ by the left adjoint action, and $\g_1\cong \g_{-1}^*$ as $\g_0$ modules, so if $A\in \g_0$ and $\phi\in \g_1$ then $[A,\phi]=-\phi\circ A$. Fix $\psi\in \g_{-2}$ and $\phi\in \g_2$ such that $\langle \psi, \phi \rangle=1$. Define $L:\g_{-1}\times \g_{-1}\to \R$ by $L(X,Y)=\phi([X,Y])$. Let $\flat:\g_{-1}\to \g_1$ be given by $X^{\flat}(Y)=L(X,Y)$ and let $\sharp:\g_1\to \g_{-1}$ be its inverse.
\begin{proposition}
\begin{enumerate}[(a)]
    \item $[\phi,\psi]=E$.
    \item For $\alpha\in \g_1$, $[\psi, \alpha]=\alpha^{\sharp}$.
    \item For $X\in \g_{-1}$, $[ \phi,X]=X^{\flat}$.
    \item For $\alpha,\beta\in \g_1$, $[\alpha,\beta]=\beta(\alpha^{\sharp})\phi$.
\end{enumerate}
\end{proposition}
\begin{proof}
\begin{enumerate}[(a)]
\item
    Let $X\in \g_0^{ss}$. Then
    \[
    \langle X,[\phi,\psi]\rangle
    =\langle [X,\phi], \psi \rangle
    =0
    \]
    because $\g_0^{ss}=[\g_0,\g_0]$ acts trivially on the $1$-dimensional $\g_2$. It follows that 
    \[
    [\phi,\psi]\in (\g_0^{ss})^{\perp}=\z(\g_0)\subset \h.
    \]
    Next, let $H\in \h$. Then
    \begin{align*}
    \langle H, [\phi, \psi] \rangle
    &=\langle [H, \phi], \psi \rangle\\
    &=\mu(H) \langle \phi, \psi \rangle\\
    &=\mu(H)\\
    &=\langle H, H_{\mu} \rangle.
    \end{align*}
    It follows from the nondegeneracy of our invariant bilinear form on $\h$ that $[\phi,\psi]=H_{\mu}=E$.
    \item Using $[\psi,X]=0$, 
    \begin{align*}
    L([\psi,\alpha],X)
    &=\langle \phi, [[\psi,\alpha],X]\rangle\\
    &=\langle \phi, [\psi, [\alpha, X]] \rangle\\
    &=\langle [\phi, \psi], [\alpha,X] \rangle\\
    &=\langle E,[\alpha, X]\rangle\\
    &=\langle [E,\alpha], X \rangle\\
    &=\langle \alpha, X \rangle.
    \end{align*}
    \item Let $\alpha\in \g_1$. Using $[\alpha,\phi]=0$,
    \begin{align*}
  [\phi, \alpha^{\sharp}]
  &=[\phi, [\psi, \alpha]]\\
  &=[[\phi,\psi], \alpha]\\
 &=[E,\alpha]\\
  &=\alpha.
  \end{align*} 
  \item 
\[
\langle \psi, [\alpha, \beta]\rangle
=\langle [\psi, \alpha], \beta \rangle
=\beta(\alpha^{\sharp}).
\]
    \end{enumerate}
\end{proof}

Another important point is that when computing algebraic brackets of sections of $\G_0\times_{G_0}\g$ we will always work in a gauge in which $r$ corresponds to $\psi\in \g_{-2}$ and $\theta$ corresponds to $\phi\in \g_2$.
\subsection{Ricci Calculus}
We assume the Einstein summation convention throughout. Recall that trace commutes with covariant derivative on any vector bundle. For example, if $T$ is a tensor preserving $H$ then $\nabla_0T_{i}^{\ i}=\nabla_0(T_i^{\ i})$ even though $i$ only varies over a basis for $H$. 

One has to be careful when using an alternating tensor like $L_{ij}$ to raise and lower indices as opposed to a symmetric tensor like a Riemannian metric. Let $L^{ij}$ be the inverse matrix to $L_{ij}$. To lower an index, let $T_j=L_{ij}T^i$. To raise an index, let $T^j=L^{ij}T_i$. Note that $L^{ij}$ is not the result of starting with $L_{ij}$ and raising both indices, since this would give $-L^{ij}$. The trace of $T$ satisfies
\[
T_i^{\ i}
=L^{ji}T_{ij}
=-L^{ij}T_{ij}
=-T^j_{\ j}
=-T^i_{\ i}.
\]
Given an endomorphism $T:H\to H$, we say that it preserves $L$ if
\[
L(T(X),Y)+L(X,T(Y))=0
\]
for $X,Y\in H$. An endomorphism $T_i^{\ j}$ preserves $L$ as an endomorphism if and only if $T_{ij}=T_{ji}$. To see why, suppose that $T$ preserves $L$. Then
\[
0=L(T(E_i),E_j)+L(E_i,T(E_j))
=L(T(E_i),E_j)-L(T(E_j),E_i)
=T_{ij}-T_{ji}.
\]
The converse is similar. We compute some of the quantities from the previous subsection in coordinates. If $X\in \g_{-1}$, then
\[
(X^{\flat})_j=X^{\flat}(X_j)=L(X,X_j)=L_{ij}X^i.
\]
Therefore, if $\alpha\in \g_1$ then
\[
(\alpha^{\sharp})^j=L^{ij}\alpha_i.
\]
Finally, if $\alpha,\beta\in \g_1$ then
\[
[\alpha,\beta]
=\beta(\alpha^{\sharp})\phi
=\beta_j(\alpha^{\sharp})^j\phi
=L^{ij}\alpha_i \beta_j\phi.
\]
Next,
\[
T((X_i)^{\flat})=X_i^{\flat}(T)=L(X_i,T)=-T_i.
\]
Likewise,
\[
T((X^i)^{\sharp})
=T_j((X^i)^{\sharp})^j
=T_jL^{kj}(X^i)_k
=L^{ij}T_j
=-T^i.
\]
  
\section{Rho Tensor}
\subsection{Weyl Torsion}
For $X\in H$, recall that the Weyl torsion is defined by $\tau(X)=\nabla_rX-[r,X]$. This is related to the torsion in homogeneity $2$ by $\tau(X)=T^{(2)}(r,X)$. By Proposition~\ref{prop::PreservesLevi}, the Weyl connection $\nabla$ preserves $\L$. Since $\nabla$ preserves $\theta:Q\to \R$, it must also preserve $L=\theta\circ \L$.
\begin{proposition}
The endomorphism $\tau$ preserves  $L:H\times H\to \R$.
\end{proposition}
\begin{proof}
Combining $\nabla L=0$ with the Jacobi identity,
\begin{align*}
L(\tau(X),Y)+L(X,\tau(Y))
&=L(\nabla_r X-[r,X],Y)+L(X,\nabla_r Y-[r,Y])\\
&=rL(X,Y)-\theta([[r,X],Y]+[X,[r,Y]])\\
&=rL(X,Y)-\theta([r,[X,Y]])\\
&=r\theta([X,Y])-[X,Y]\theta(r)-\theta([r,[X,Y]])\\
&=d\theta(r,[X,Y])\\
&=0.
\end{align*}
\end{proof}
This proposition is equivalent to the assertion that $\tau$ preserves the Levi bracket $\L:H\times H\to Q$.
Immediate consequences are that the Weyl torsion $\tau$ is traceless, that is $\tau_i^{\ i}=0$, and the index-lowered version is symmetric, that is $\tau_{ij}=\tau_{ji}$.
\subsection{Contracted Bianchi Identity}
In this section we prove a curvature identity that is well-known for CR manifolds but which actually generalizes to all torsion-free parabolic contact geometries. Several other well-known curvature identities for CR manifolds can be proven in this generality by tracing either the algebraic or differential Bianchi identities.

Let $\widetilde{T}(X,Y)=\nabla_XY-\nabla_YX-[X,Y]$ be the affine connection torsion. Then $\widetilde{T}(X,Y)=T(X,Y)-\{X,Y\}$. Since $\widetilde{T}^{(1)}=0$ we have that $\widetilde{T}_{ij}^{\ \ k}=\widetilde{T}_{0i}^{\ \ 0}=0$. Equivalently, if $X,Y\in H$ then $\widetilde{T}(X,Y)\in \langle r \rangle$ and $\widetilde{T}(r,X)\in H$. The notation $\sum_{XYZ}$ denotes the sum over cyclic permutations of the variables $X,Y,Z$. The following identity is well-known and may be found, for example, in \cite{KobayashiNomizu}.
\begin{proposition}[Algebraic Bianchi Identity]
For any affine connection $\nabla$,
\[
\sum_{XYZ}R(X,Y)Z=\sum_{XYZ}\nabla_X \widetilde{T}(Y,Z)+\widetilde{T}(\widetilde{T}(X,Y),Z).
\]
\end{proposition}

\begin{proposition}[Contracted Bianchi Identity]\label{prop::ContractedBianchi}
\[
R_{i0j}^{\ \ \ i}=\nabla_i\tau_{j}^{\ i}.
\]
\end{proposition}
\begin{proof}
Apply the algebraic Bianchi identity to the Weyl connection $\nabla$. As a consequence of Corollary~\ref{cor::ExactWeyl}, if $X,Y\in TM$ then $R(X,Y)$ acts trivially on $Q$. Therefore $R(X,Y)$ preserves $L$, so is traceless as an endomorphism of $H$. Let $Y=r$ and $Z\in H$. Tracing the left hand side of the Bianchi identity over $X$ yields
    \[
    (R_{i0j}^{\ \ \ i}+R_{00j}^{\ \ \ 0})+(R_{0ji}^{\ \ \ i}+R_{0j0}^{\ \ \ 0})
    +(R_{ji0}^{\ \ \ i}+R_{j00}^{\ \ \ 0})
    =R_{i0j}^{\ \ \ i}
    \]
    since all terms except the first vanish. The term with the covariant derivative is equal to
    \[
    (\nabla_i\widetilde{T}_{0j}^{i}+\nabla_0\widetilde{T}_{0j}^{0})
    +(\nabla_0\widetilde{T}_{ji}^i+\nabla_0\widetilde{T}_{j0}^0)
    +(\nabla_j\widetilde{T}_{i0}^{i}+\nabla_i\widetilde{T}_{00}^{0}).
    \]
    Since $\nabla r=0$, the first term is equal to $\nabla_i \tau_j^{i\ }$. Also, $\nabla_j \widetilde{T}_{i0}^{\ \ i}=-\nabla_j \tau_i^{\ i}=0$. It is straightforward to see that all other terms vanish. To calculate the final term in the Bianchi identity, first suppose $X=r$. Then
    \[
    \widetilde{T}(\widetilde{T}(r,r),Z)+\widetilde{T}(\widetilde{T}(Z,r),r)+\widetilde{T}(\widetilde{T}(r,Z),r)=0
    \]
    since $\widetilde{T}$ is alternating.
    Otherwise, if $X\in H$ then
    \[
    \widetilde{T}(\widetilde{T}(X,r),Z)+\widetilde{T}(\widetilde{T}(Z,X),r)+\widetilde{T}(\widetilde{T}(r,Z),X)\in \langle r \rangle
    \]
    so this contributes nothing to the trace.
\end{proof}
\subsection{Homogeneity $2$}
The lowest homogeneity component that occurs in the Rho tensor $\P$ corresponds to maps $\g_{-1}\to \g_1$ of homogeneity $2$. The highest homogeneity component corresponds to maps $\g_{-2}\to \g_2$ of homogeneity $4$. To determine $\P$ in homogeneity $2$ we apply a shortcut. Let $X\in H$. Since $\kappa$ is assumed torsion-free, the torsion component $\kappa^{(2)}(r,X)$ vanishes. Consequently,
    \begin{align*}
0
&=\kappa^{(2)}(r,X)\\
&=(T^{(2)}+\partial \P^{(2)})(r,X)\\
&=T^{(2)}(r,X)+\{r,\P^{(2)}(X)\}\\
&=T^{(2)}(r,X)+\P^{(2)}(X)^{\sharp}.
    \end{align*}
Therefore
\[
\P^{(2)}(X)
=-T^{(2)}(r,X)^{\flat}
=-\tau(X)^{\flat}.
\]
That is,
\[
\P_{ij}=-\tau_{ij}.
\]

\subsection{Homogeneity $3$}
\begin{remark}
\cite{CapGuo} suggests an alternate strategy leading to a proof that $\P_{0i}=2\P_{i0}$ while avoiding the contracted Bianchi identity proven above. An analysis of the harmonic curvature components determined by Kostant's Borel--Weil--Bott shows that the output is always in a root space, and is therefore orthogonal to the Cartan subalgebra. In the absence of torsion, this condition is $P$-invariant, and so the improved Bianchi identity stated in \cite{Guo} shows that the output of curvature is always orthogonal to the Cartan subalgebra and, in particular, to $E$. By our choice of bundle of scales and the exactness of the Weyl structure, the output of $R$ is also orthogonal to $E$. Considering equation \eqref{eqn::WeylCurvature} gives $0=\langle E,\partial \P^{(3)}(r,X)\rangle$ which leads to $\P_{0i}=2\P_{i0}$. 
\end{remark}
Since $\kappa$ is torsion-free, the homogeneity one component $T^{(1)}=\kappa^{(1)}$ of the torsion tensor vanishes. In particular, $T(X,Y)=0$ for $X,Y\in H$ because this torsion component necessarily has homogeneity $0$ or $1$. However, $T^{(2)}$ is not in general equal to $\kappa^{(2)}$, which is why this term corresponding to the Weyl torsion $\tau$ may be nonvanishing. In addition, the $\{\P(X),\P(Y)\}$ term that appears in Proposition~\ref{prop::CottonYork} can only be valued in $Q$ corresponding to $\g_2$. Then
    \begin{align*}
    Y^{(3)}_{ijk}
    &=\nabla_{i}\P_{jk}-\nabla_{j}\P_{ik}\\
    &=-\nabla_{i}\tau_{jk}+\nabla_j\tau_{ik}.
    \end{align*}
Next, $R^{(3)}$ is of the form $Q\times H\to \End_0(H)$, corresponding to maps $\g_{-2}\times \g_{-1}\to \g_0$. By Corollary~\ref{cor::ExactWeyl}, the output in $\g_0$ is valued in the part that acts trivially on $\g_2$. Therefore in the Kostant codifferential, the contraction between the $\g_0$ part and the $\g_{-2}^*\cong \g_2$ part vanishes. Then $\partial^*R^{(3)}$ corresponds to maps $\g_{-2}\to \g_{1}$, meaning that it is a map $Q\to H^*$. Then for $X\in H$,
\[
\partial^*R^{(3)}(X)=0.
\]
The tensor $Y^{(3)}$ is a map $H\times H\to H^*$ corresponding to maps $\g_{-1}\times \g_{-1}\to \g_1$. To compute $\partial^*Y^{(3)}_{i0}$, the $H\to Q^*$ part of the Kostant codifferential corresponding to $\g_{-1}\to \g_2$, the relevant term in the Kostant codifferential is the one corresponding to contraction of the output $\g_1$ with the inputs $\g_{-1}^*\cong \g_1$. Using the fact that the Weyl torsion $\tau$ is traceless,
    \begin{align*}
    \partial^*Y^{(3)}_{i0}
    &=\{X^j,Y^{(3)}(X_i,X_j)\}_0\\
    &=L^{jk}Y_{ijk}\\
    &=Y_{i\ j}^{\ j}\\
    &=-\nabla_i\tau^{j}_{\ j}+\nabla^j\tau_{ij}\\
    &=\nabla^j\tau_{ij}.
    \end{align*}
Since $\P^{(3)}_{i0}$ is a map $H\to Q^*$ corresponding to maps $\g_{-1}\to \g_2$, Proposition~\ref{prop::KostantEigenvalues} says that $\square$ acts by $-(2n+1)$. Then
\[
\P^{(3)}_{i0}
=-\square^{-1}\partial^*(R^{(3)}+Y^{(3)})_{i0}
=\frac{1}{2n+1}\nabla^j\tau_{ij}.
\]
The homogeneity $3$ component of the curvature tensor $R^{(3)}$ is a map $Q\times H\to \End_0(H)$ corresponding to maps $\g_{-2}\times \g_{-1}\to \g_0$. To compute $\partial^*R^{(3)}_{0i}$, the relevant part of the Kostant codifferential contracts the input $\g_{-1}^*\cong \g_1$ with the output $\g_0$. Using Proposition~\ref{prop::ContractedBianchi},
\begin{align*}
\partial^*R^{(3)}_{0j}
&=\{X^i,R^{(3)}(r,X_i)\}_j\\
&=X^i\circ R^{(3)}(r, X_i)(X_j)\\
&=R_{0ij}^{\ \ \  i}\\
&=-R_{i0j}^{\ \ \  i}\\
&=-\nabla_i\tau_{j}^{\ i}\\
&=\nabla^i \tau_{ji}.
\end{align*}
Next,
\begin{align*}
    \partial^*Y^{(3)}_{0j}
    &=-\frac{1}{2}Y^{(3)}(\{X^i,r\},X_i)(X_j)\\
    &=\frac{1}{2}Y^{(3)}((X^i)^{\sharp},X_i)(X_j)\\
    &=\frac{-1}{2}(Y^{(3)})^i_{\ ij}\\
    &=\frac{1}{2}(\nabla^i\tau_{ij}-\nabla_i\tau^i_{\ j})\\
    &=\nabla^i\tau_{ij}.
\end{align*}
It follows that
\begin{align*}
\partial^*(R^{(3)}+Y^{(3)})_{0j}=2\nabla^i\tau_{ij}.
\end{align*}
Since $\P^{(3)}_{0i}$ is a map $Q\to H^*$ corresponding to maps $\g_{-2}\to \g_{1}$, Proposition~\ref{prop::KostantEigenvalues} says that $\square$ acts by $-(2n+1)$. Then
\begin{align*}
\P^{(3)}_{0i}
&=-\square^{-1}\partial^*(R^{(3)}+Y^{(3)})_{0i}\\
&=\frac{2}{2n+1}\nabla^j\tau_{ij}.
\end{align*}

\subsection{Homogeneity $4$}
To compute the Cotton--York tensor component $Y_{0ij}^{(4)}$, observe that $T_{0i}^{\ \ k}=\tau_i^{\ k}$. Then
\begin{align*}
    Y^{(4)}_{0ij}
    &=\nabla_0\P_{ij}
    -\nabla_{i}\P_{0j}
    +\tau_i^{\ k}\P_{kj}\\
    &=-\nabla_0 \tau_{ij}
    -\frac{2}{2n+1}\nabla_{i}\nabla^k\tau_{kj}
    -\tau_i^{\ k}\tau_{kj}.
\end{align*}
Next,
\begin{align*}
Y^{(4)}_{ij0}
&=\nabla_i\P_{j0}-\nabla_j\P_{i0}+\{\P^{(2)}(X_i),\P^{(2)}(X_j)\}_0\\
&=\nabla_i \P_{j0}-\nabla_j\P_{i0}
+L^{kl}\tau_{ik}\tau_{jl}\\
&=\frac{1}{2n+1}(\nabla_{i}\nabla^k\tau_{kj}
-\nabla_{j}\nabla^k\tau_{ki})+(\tau\circ \tau)_{ij}.
\end{align*}
Now that we know the homogeneity $4$ component of the Cotton--York tensor, we can compute its Kostant codifferential, necessarily a map $Q\to Q^*$ corresponding to maps $\g_{-2}\to \g_2$. We evaluate
\begin{align*}
    \partial^*Y^{(4)}_{00}
    &=\{X^i,Y^{(4)}(r,X_i)\}
    -\frac{1}{2}\sum Y^{(4)}(\{X^i,r\},X_i).
\end{align*}
Because $\tau$ is traceless,
\begin{align*}
\{X^i,Y^{(4)}(r,X_i)\}
&=L^{ij}Y_{0ij}\\
&=-Y_{0i}^{\ \ i}\\
&=\nabla_0\tau_i^{\ i}
+\frac{2}{2n+1}\nabla_i\nabla^j\tau_j^{\ i}
+\tau_i^{\ k}\tau_k^{\ i}\\
&=\frac{-2}{2n+1}\nabla^i\nabla^j\tau_{ij}
+\tr(\tau\circ \tau).
\end{align*}
To compute the second term,
\begin{align*}
Y^{(4)}(\{X^i,r\},X_i)_0
&=-Y^{(4)}((X^i)^{\sharp},X_i)_0\\
&=(Y^{(4)})^i_{\ i0}\\
&=\frac{1}{2n+1}(\nabla^i\nabla^j\tau_{ji}-\nabla_i\nabla^j\tau_j^{\ i})+(\tau\circ\tau)^{i}_{\ i}\\
&=\frac{2}{2n+1}\nabla^i\nabla^j\tau_{ij}-\tr(\tau\circ \tau).
\end{align*} 
Combining these two terms,
\begin{align*}
    \partial^*Y^{(4)}(r)
    &=\frac{-3}{2n+1}\nabla^i\nabla^j\tau_{ij}
    +\frac{3}{2} \tr(\tau \circ \tau).
\end{align*}
 There is no $R^{(4)}$ term because it would necessarily be an alternating map $Q\times Q\to \End_0(H)$ and $Q$ is $1$-dimensional. Since $\P^{(4)}$ is a map $Q\to Q^*$ corresponding to maps $\g_{-2}\to \g_2$, Proposition~\ref{prop::KostantEigenvalues} says that $\square$ acts by $-3n$. Therefore
\[
\P^{(4)}_{00}
=-\square^{-1}\partial^*Y^{(4)}_{00}
=\frac{1}{n}\left(\frac{-1}{2n+1}\nabla^i\nabla^j\tau_{ij}+\frac{1}{2}\tr(\tau\circ \tau)\right),
\]
finishing the proof of Theorem~\ref{thm::RhoTensor}.

\section{Tanaka--Webster Connection}
\subsection{Ricci Tensor}
In our discussion of the Tanaka--Webster connection we attempt to build on the approach in \cite{CapSlovak2} but use some slightly different conventions. While that text uses the identification $\partial:\g_0\to \End(\g_-)$, the right action of $\g_0$ on $\g_-$, we found it more convenient to identify $\g_0$ with endomorphisms of $\g_-$ by its left action. Then for $A\in \g_0$ we can write $\partial A=-A$. 

Our convention for the Ricci element $\widehat{\Ric}\in \g_0$ also differs from the one adopted in \cite{CapSlovak2} by a minus sign. Thus, for $\{\cdot,\cdot\}:\bigwedge^2 \g_{-1}^*\cong \bigwedge^2 \g_1\to \g_2\cong \g_{-2}^*$ induced by the bracket, we define 
\[
\widehat{\Ric}=-\{\cdot,\cdot\}\widehat{R}^{(2)}(r)=\partial^*\widehat{R}^{(2)}(r)
\]
where $\widehat{R}^{(2)}\in \bigwedge^2 \g_{-1}^*\otimes \g_0$ is the curvature of $\widehat{\nabla}$ in homogeneity $2$. Because $\widehat{\nabla}r=0$, the curvature tensor $\widehat{R}$ is valued in the part of $\End_0(H)$ acting trivially on $Q$. Thus $\widehat{\Ric}$ acts trivially on $Q$. It follows that $\widehat{\Ric}$ preserves the Levi bracket on $H$ and, in particular, acts tracelessly on $H$. Acting on $H$ from the left,
\begin{align*}
\widehat{\Ric}_j^{\ k}
&=\partial^*\widehat{R}(r)_{j}^{\ k}\\
&=-\frac{1}{2}\widehat{R}(\{X^i,r\},X_i)\\
&=\frac{1}{2}\widehat{R}((X^i)^{\sharp},X_i)_j^{\ k}\\
&=-\frac{1}{2}\widehat{R}^{i \ \ k}_{\ i j}\\
&=\frac{1}{2}\widehat{R}_{i \ j}^{\ i \ k}.
\end{align*}

\subsection{Definition}
While it is possible to write down defining conditions for the Weyl connection $\nabla$ associated to underlying data, these conditions are generally not as simple as those defining the Tanaka--Webster connection.
\begin{remark}
    It is shown in \cite{CapSlovak2} that it is possible to write explicit formulas for the Tanaka--Webster connection in integrable Legendrian contact, CR, and contact projective geometries. These structures are related to the $A_l$ and $C_l$ series of simple Lie algebras. It is interesting to ask whether something similar is possible in torsion-free Lie contact geometries, related to the $B_l$ and $D_l$ series. Such geometries are necessarily flat.
\end{remark}

 To define the Tanaka--Webster connection, we will present a version of Proposition 5.2.12 of \cite{CapSlovak2} while reframing it slightly. Consider the soldering form $\omega_-$ on $\G_0$ induced by a contact form on a normal parabolic contact structure. Any principal connection $\gamma$ on $\G_0$ 
induces a torsion tensor $d\omega_-+[\gamma,\omega_-]$
relative to the soldering form. This induces a $2$-form valued in $\G_0\times_{G_0}\g_-\cong Q\oplus H \cong TM$.
If the affine connection induced by that principal connection is denoted $\widehat{\nabla}$, the induced torsion tensor is given by $\widehat{T}(X,Y)=\widehat{\nabla}_X Y-\widehat{\nabla}_Y X -[X,Y]+\{X,Y\}$. The difference between two principal connections $\gamma$ and $\gamma'$ is a horizontal $1$-form valued in $\g_0$, which therefore has strictly positive homogeneity. Therefore a different choice of principal connection only affects the torsion $\widehat{T}$ in positive homogeneity.

In particular, $\widehat{T}$ agrees with $T$, the torsion induced by the Weyl connection in nonpositive homogeneity. So for $i\le 0$ we have $\widehat{T}^{(i)}=T^{(i)}=0$. Suppose now that $\widehat{\nabla}\theta=0$. Since $\widehat{\nabla}$ is associated to a principal connection for $\G_0$, it automatically preserves $\ker \theta=H$. For the same reason $\widehat{\nabla}$ must preserve the line spanned by $r$ so compatibility with $\theta$ is equivalent to $\widehat{\nabla}r=0$. For $X\in H$, observe that
\[
0
=d\theta(r,X)
=r\theta(X)-X\theta(r)-\theta([r,X])
=-\theta([r,X]),
\]
so $[r,X]\in H$. It follows that
\[
\widehat{T}(r,X)
=\widehat{\nabla}_r X-\widehat{\nabla}_Xr-[r,X]\in H,
\]
so the $Q\times H\to Q$ component of $\widehat{T}$ vanishes. Therefore the entire homogeneity $1$ component of $\widehat{T}$ is $\widehat{T}^{(1)}|_{H\times H}$ given by 
\[
\widehat{T}^{(1)}(X,Y)=\widehat{\nabla}_X Y-\widehat{\nabla}_Y X-\pi([X,Y]).
\]

Now we consider the homogeneity $2$ component of $\widehat{T}$. Because $\widehat{T}$ is a $2$-form we have $\widehat{T}(r,r)=0$, so the $Q\times Q\to Q$ component vanishes. Therefore the entire homogeneity $2$ part is given by the $Q\times H\to H$ component of $\widehat{T}$. If we define $\widehat{\tau}(X)=\widehat{\nabla}_r X-[r,X]$, then $\widehat{\tau}(X)=\widehat{T}^{(2)}(r,X)$.

Finally, observe that it is a consequence of Kostant's Borel--Weil--Bott that $(\ker \square)_0\subset \g_{-1}^*\otimes \g_{-1}$, meaning that harmonic $1$-cochains in homogeneity $0$ correspond to maps $H\to H$. In particular, the $Q\to Q$ component is always $0$. We give the following reformulation of Proposition 5.2.12 of \cite{CapSlovak2}.
\begin{proposition}[Tanaka--Webster Connection]
Fix a normal parabolic contact geometry $(\G,\omega)$ and a contact form $\theta$ inducing a soldering form. There exists a unique affine connection on $TM$ induced by a principal connection on $\G_0$ for which
\begin{itemize}
\item $\widehat{\nabla}\theta=0$,
\item $\widehat{T}^{(1)}|_{H\times H}\in \ker \square \subset \bigwedge^2 H^* \otimes H$,
\item $\widehat{\tau}\in \ker \square\subset H^*\otimes H$.
\end{itemize}
Furthermore, if $(\G,\omega)$ is torsion-free then $\widehat{T}^{(1)}=0$.
\end{proposition}
The last statement is the only part that requires some additional justification. As a consequence of Theorem 5.2.11 of \cite{CapSlovak2}, $\widehat{T}\in \ker \square$ determines $\widehat{\nabla}$ uniquely in contact directions and, in particular, it agrees with the Weyl connection $\nabla$ in these directions. Then
\[
\widehat{T}^{(1)}=T^{(1)}=\kappa^{(1)}=0
\]
because the Cartan geometry is torsion-free. The quantity $\widehat{\tau}$ associated to the Tanaka--Webster connection $\widehat{\nabla}$ is known as the \emph{Webster torsion}.

\subsection{Relation to the Weyl Connection}
Let $\widehat{\nabla}$ be the affine Tanaka--Webster connection. As discussed in \cite{CapSlovak2}, $\nabla$ agrees with $\widehat{\nabla}$ when restricted to directions in $H$. Also, $\nabla r=\widehat{\nabla} r=0$. Let $X\in H$. The difference between the two connections in the Reeb direction can be expressed in terms of a correction term $\rho$, so that $\nabla_rX=\widehat{\nabla}_rX-\rho(X)$. If we take $\widehat{\nabla}$ to be fixed, then $\rho$ determines $\nabla$. We apply a recipe in \cite{CapSlovak2} to compute the correction term $\rho$ in each case. This correction term is related to the Ricci tensor and we will call it the modified Ricci tensor. In the context of CR geometry $\rho$ is also called the Schouten tensor. This should not be confused with the Rho tensor $\P$, which is also often called the Schouten tensor, although $\rho$ does represent a component of $\P$. 

 The bracket $\{\}:\bigwedge^2 \g_{1}\to \g_{2}$ induces a map $\{\cdot,\cdot\}:\bigwedge^2 \g_{-1}^*\otimes \g_{-2}\to \g_{-2}^*\otimes \g_{-2}$. Recall that the \emph{Levi bracket} $\L:\g_{-1}\times \g_{-1}\to \g_{-2}$ is defined as the restriction of the Lie bracket. Define $s$ by $\{\cdot,\cdot\}\L=s\Id_{\g_{-2}}$. Recall that $\End_0(H)=\G_0\times_{G_0}\g_0$. The following result is equivalent to Theorem 5.2.13 in \cite{CapSlovak2}.

\begin{theorem}
There exists $\rho\in \End_0(H)$ such that $\nabla_r X=\widehat{\nabla}_rX +\partial \rho(X)$, and it satisfies $-(s+\square)\rho =\widehat{\Ric}$.
\end{theorem}
As a consequence of the identification of $\g_0$ with its action on $\g_-$ from the left, we can write $\partial \rho=-\rho$. Then the result in the theorem can be rephrased as
\[
\nabla_r X=\widehat{\nabla}_rX-\rho(X).
\]

We can compute $s$ explicitly under our choice of normalization condition on the bilinear form.
Since
\begin{align*}
\{\cdot,\cdot\}(\L)(\psi)
&=\frac{1}{2}\L([X^i,\psi],X_i)\\
&=-\frac{1}{2}L((X^i)^{\sharp},X_i)\psi\\
&=-\frac{1}{2}X^i(X_i)\psi\\
&=-n\psi,
\end{align*}
we get $s=-n$. Plugging this into the theorem,
$
(n-\square)\rho=\widehat{\Ric}.
$
We will see through the case analysis later in this section that $(n-\square)$ acts invertibly in each case, so we can evaluate 
\begin{equation}\label{eqn::Laplacianrho}
\rho=(n-\square)^{-1}(\widehat{\Ric}).
\end{equation}
The Kostant Laplacian acts on components falling into different isotypic subrepresentations by different scalars. Because of this, we will call $\rho$ the \emph{modified Ricci tensor}. We will see that the modified Ricci tensor $\rho$ also agrees with the Schouten tensor as defined in CR geometry, so it could also be appropriate to call $\rho$ the Schouten tensor.

\subsection{Weyl Torsion Decomposition}
Because $\tau$ has homogeneity $0$, the $2$-cochain $\partial \tau$ has homogeneity $0$. Therefore $\partial \tau (r,X)$ for $X\in H$ has homogeneity $-3$ and necessarily vanishes. However, for $X,Y\in H$,
\begin{align*}
    \partial \tau(X,Y)
    &=\{X,\tau(Y)\}-\{Y,\tau(X)\}-\tau(\{X,Y\})\\
    &=\L(X,\tau(Y))+\L(\tau(X),Y)\\
    &=0
\end{align*}
since $L$ is antisymmetric and $\tau$ preserves the Levi bracket $\L$. Therefore $\tau\in \ker \partial$. The algebraic Hodge decomposition says that $\ker \partial=\ker \square \oplus \im \ \partial$. On the other hand,
\[
\tau(X)
=\nabla_rX-[r,X]
=\widehat{\nabla}_r X-[r,X]-\rho(X)
=\widehat{\tau}(X)-\rho(X).
\]
By definition of the Tanaka--Webster connection, $\widehat{\tau}\in \ker \square$. By construction, $-\rho=\partial \rho\in \im \ \partial$, so 
\[
\tau=\widehat{\tau}-\rho
\]
is exactly the algebraic Hodge decomposition of $\tau$. Since they are contained in $\ker \partial$, the tensors $\widehat{\tau}$ and the modified Ricci tensor $\rho$ must also preserve $L$, so they are also traceless and their index-lowered versions are also symmetric. In the context of CR geometry, a contact form is called a pseudohermitian structure. If $\widehat{\tau}=0$, the pseudohermitian structure is \emph{Sasakian}. The modified Ricci tensor $\rho=(n-\square)^{-1}\widehat{\Ric}$ vanishes if and only if $\widehat{\Ric}$ vanishes. It follows that $\tau=0$ if and only if the pseudohermitian structure is Sasakian and Ricci-flat.

\subsection{Modified Ricci Tensor}
To prove Theorem~\ref{thm::rhoTensor}, we will apply equation \eqref{eqn::Laplacianrho} on a case-by-case basis depending on the complexification $\g^{\C}$. The relevant eigenvalues of the Kostant Laplacian are collected in Lemma~\ref{lemma::LaplacianEigenvalues}. 

Recall that the simple ideals in $\g_0^{ss}$ correspond to the connected components of the Dynkin diagram after removing the crossed nodes inducing the grading. The crossed nodes corresponding to the contact gradings are the ones shown in the figure in Section~\ref{app::0Cochains}. Therefore $\g_0^{ss}$ has $3$ simple ideals in the $D_4$ case, two simple ideals in the $B_l$ for $l\ge 3$ case or the $D_l$ for $l\ge 5$ case, and one simple ideal in the other cases. In the $B_l$ and $D_l$ cases we will decompose $\g_0^{ss}=\g_0'\oplus \g_0''$ where $\g_0'$ corresponds to the simple root $\alpha_1$ and $\g_0''$ corresponds to the simple roots $\alpha_i$ for $i\ge 3$. In particular, in the $D_4$ case $\g_0''$ has two simple ideals.

Recall that in the $A_l$ case, the complexification $\g_{-1}^{\C}$ decomposes as the direct sum of two $\g_0^{\C}$-irreducible representations, $\g_{-1}^{\E}$ and $\g_{-1}^{\F}$.
The space $\g_{-1}^{\E}$ is the direct sum of root spaces $\g_{\alpha}$ such that $\alpha=\alpha_1+\cdots+\alpha_i$ for $i<l$, and $\g_{-1}^{\F}$ is the direct sum of root spaces $\g_{\alpha}$ such that $\alpha=\alpha_i+\cdots+\alpha_l$ for $1<i$. Since these subspaces may be swapped by $G_0$, we have to fix a local section of $\G_0\to M$ and define $\E$ and $\F$ to be subbundles of $\G_0\times_{G_0}\g_{-1}$ corresponding to $\g_{-1}^{\E}$ and $\g_{-1}^{\F}$ relative to that local section. Recall that $\widehat{R}$ is defined to be the trace of $\widehat{\Ric}$ acting on $\E$, and $\iota:H\to H$ is the map acting by $+1$ on $\E$ and $-1$ on $\F$. We have established that $\widehat{\Ric}$ acts tracelessly on $H$. Therefore $\widehat{\Ric}$ acts on $\F$ with trace $-\widehat{R}$. This implies that $\widehat{R}\iota$ is in fact independent of a choice of ordering of $\E$ and $\F$, and thus is independent of a choice of local section of $\G_0$.

\begin{thmrhoTensor}[Modified Ricci Tensor]
Given a normal parabolic contact geometry with a fixed contact form, the modified Ricci tensor $\rho$ can be computed by the following.
\begin{enumerate}[(1)]
\item If $\g^{\C}\cong A_l$ for $l\ge 2$, $\rho=\frac{1}{n+2}\left(\widehat{\Ric}-\frac{\widehat{R}}{2(n+1)}\iota\right)$.
\item If $\g^{\C}\cong B_l$ for $l\ge 3$ or $\g^{\C}\cong D_l$ for $l\ge 4$, $\rho=\frac{\widehat{\Ric}'}{2n}+\frac{\widehat{\Ric}''}{n+4}$.
\item If $\g^{\C}\cong C_l$ for $l\ge 2$, $\rho=\frac{\widehat{\Ric}}{n+1}$.
\item If $\g^{\C}\cong E_6$, $\rho=\frac{\widehat{\Ric}}{16}$.
\item If $\g^{\C}\cong E_7$, $\rho=\frac{\widehat{\Ric}}{24}$.
\item If $\g^{\C}\cong E_8$, $\rho=\frac{\widehat{\Ric}}{40}$.
\item If $\g^{\C}\cong F_4$, $\rho=\frac{\widehat{\Ric}}{12}$.
\item If $\g^{\C}\cong G_2$, $\rho=\frac{3\widehat{\Ric}}{16}$.
\end{enumerate}
\end{thmrhoTensor}
\begin{proof}
Recall that $\widehat{\Ric}$ is a section of $\G_0\times_{G_0}\g_0=\End_0(H)$ and $\rho$ is a section of $\End_0(H)$ given by $(n-\square)^{-1}\widehat{\Ric}$. This action of $\square$ on $\g_0$ is unchanged if we consider $\g_0\subset \g_0^{\C}$ and apply the complex Kostant Laplacian. Therefore we may regard $\widehat{\Ric}$ as a section of $\G_0\times_{G_0} \g_0^{\C}$ and evaluate $(n-\square)^{-1}(\widehat{\Ric})$ there. We established that $\widehat{\Ric}$ acts trivially on $Q$, so it corresponds to elements in $\g_0$ acting trivially on $\g_{-2}$. The derived subalgebra $\g_0^{ss}$ acts trivially on $\g_{-2}$ since it is $1$-dimensional. On the other hand, the grading element $E\in \z(\g_0)$ acts nontrivially on $\g_{-2}$. It follows that $\widehat{\Ric}$ is contained in a codimension-one subspace of $\g_0$ that contains $\g_0^{ss}$. Since the dimension of $\z(\g_0)$ agrees with the number of crossed nodes, all cases except $\g^{\C}\cong A_l$ have $\dim(\z(\g_0))=1$, which implies $\widehat{\Ric}$ is a section of $\G_0\times_{G_0}\g_0^{ss}$. 
    \begin{itemize}
    \item [(1)] 
We defined $\E,\F$ and $\widehat{R}$ above. Using the $G_0$-invariant decomposition $\g_0=\g_0^{ss}\oplus \z(\g_0)$, let $\Ric_0$ be the tensor corresponding to the $\z(\g_0)$ component of $\widehat{\Ric}$. For the moment, define $\iota$ to be a section of $\G_0\times_{G_0}\z(\g_0^{\C})$ given by the condition that $\frac{\widehat{R}}{n}\iota$ corresponds to the $\z(\g_0)$ component of $\widehat{\Ric}$. We will recover the definition of $\iota$ given above. Since $\g_0^{ss}$ is derived, it acts tracelessly on $\g_{-1}^{\E}$ and $\g_{-1}^{\F}$. Therefore $\frac{\widehat{R}}{n}\iota$ acts on $\E$ with trace $\widehat{R}$ and on $\F$ with trace $-\widehat{R}$. On the other hand, $\iota$ acts $\g_0$-equivariantly on the $\g_0$-irreducible representations $\g_{-1}^{\E}$ and $\g_{-1}^{\F}$, so Schur's lemma it must act by a scalar on each. Therefore $\iota$ acts on $\E$ by $+1$ and on $\F$ by $-1$. By definition, we have $\widehat{\Ric}=\widehat{\Ric}_0+\frac{\widehat{R}}{n}\iota$. By Lemma~\ref{lemma::LaplacianEigenvalues}, $n-\square$ acts by $n+2$ on $\widehat{\Ric}_0$ and $2(n+1)$ on $\iota$. Then
\begin{align*}
\rho&=(n-\square)^{-1}\widehat{\Ric}\\
&=(n-\square)^{-1}(\widehat{\Ric}_0+\frac{\widehat{R}}{n}\iota)\\
&=\frac{1}{n+2}\widehat{\Ric}_0+\frac{\widehat{R}}{2n(n+1)}\iota\\
&=\frac{1}{n+2}\left(\widehat{\Ric}-\frac{\widehat{R}}{n}\iota+\frac{(n+2)\widehat{R}}{2n(n+1)}\iota\right )\\
&=\frac{1}{n+2}\left(\widehat{\Ric} - \frac{n\widehat{R}}{2n(n+1)}\iota\right)\\
&=\frac{1}{n+2}\left(\widehat{\Ric}-\frac{\widehat{R}}{2(n+1)}\iota\right).
\end{align*}
    \item [(2)] Aside from the $D_4$ case, the decomposition $\g_0^{ss}=\g_0'\oplus \g_0''$ is canonical and thus $G_0$-invariant. This is not obvious in the $B_3$ case because $\g_0^{ss}$ is the direct sum of two simple ideals that are both real forms of $\sl_2$. Any automorphism of $\g_0$ permutes its simple ideals. However, no grade-preserving automorphism of $\g$ can switch these simple ideals since, in particular, Lemma~\ref{lemma::LaplacianEigenvalues} shows that $\square$ acts by a different scalar on each.
    
    It follows that in the $D_4$ case we should work relative to a local section as in the $A_l$ case, but otherwise we get well-defined subbundles $\G_0\times_{G_0}\g_0'$ and $\G_0\times_{G_0}\g_0''$. Decompose $\widehat{\Ric}$ into its components in these two subbundles as $\widehat{\Ric}=\widehat{\Ric}'+\widehat{\Ric}''$. Then by Lemma~\ref{lemma::LaplacianEigenvalues},
\[
\rho
=(n-\square)^{-1}\widehat{\Ric}
=\frac{\widehat{\Ric}'}{2n}
+\frac{\widehat{\Ric}''}{n+4}.
\]
    \item [(3)--(8)]
    In the remaining cases $\g_0^{ss}$ is a simple ideal and $n-\square$ acts on $(\g_0^{ss})^{\C}$ according to the values given in Lemma~\ref{lemma::LaplacianEigenvalues}.
    \end{itemize}
\end{proof}

\section{Legendrian Contact Structures}
\subsection{Lie Algebras}
This section sets up the matrix algebra that we will need for the harmonic curvature calculation. Let $\g=\sl_{n+2}$ with Cartan subalgebra the diagonal matrices and positive roots $e_i-e_j$ for $i<j$. The simple roots are $\alpha_i=e_i-e_{i+1}$ for $1\le i \le n+1$. The contact grading is associated to the parabolic subalgebra $\p$ corresponding to the subset $\{\alpha_1,\alpha_{n+1}\}$ of simple roots, since $\mu=\lambda_1+\lambda_{n+1}$. Then
\[
\sl_{n+2}
=\left\{
\begin{bmatrix}
    a & x & \gamma\\
    X & A & y^T\\
    \beta & Y^T & b
\end{bmatrix}:
a,b,\beta,\gamma\in \R, X,Y\in \R^n, x,y\in (\R^n)^*, A\in \gl_n, a+b+\tr(A)=0
\right\}
\]
and
\[
\p=
\left\{\begin{bmatrix}
a & x & \gamma\\
0 & A & y^T\\
0 & 0 & b  
\end{bmatrix}:
a,b,\gamma\in \R, x,y\in (\R^{n})^*, A\in \gl_n, a+b+\tr(A)=0
\right\}.
\]
Define a grading $\g=\g_{-2}\oplus \g_{-1}\oplus \g_0\oplus \g_1\oplus \g_2$ where the subspace $\g_{-2}$ corresponds to $\beta$, $\g_{-1}=\g_{-1}^E\oplus \g_{-1}^F$ corresponds to $X$ and $Y$, $\g_0$ corresponds to $a,A$ and $b$, $\g_1=\g_1^E\oplus \g_1^F$ corresponds to $x$ and $y$, and $\g_2$ corresponds to $\gamma$. This is a contact grading. We will write elements of $\g_{-1}$ as $(X,Y)$ and elements of $\g_1$ as $(x,y)$. The grading element is 
\[
E=
\begin{bmatrix}
 1 & 0 & 0\\
 0 & 0 & 0\\
 0 & 0 & -1
\end{bmatrix}.
\]
The trace form $\langle A,B \rangle=\tr(AB)=A_{i}^{\ j}B_j^{\ i}$ is equal to $\frac{1}{2(n+2)}$ times the Killing form, and is in fact the unique $\g$-invariant bilinear form such that $|E|^2=2$. Let $\psi\in \g_{-2}$ be the matrix corresponding to $\beta=1$ and $\phi\in \g_2$ the matrix corresponding to $\gamma=1$, so $\langle \psi, \phi \rangle=1$. Recall that $L:\g_{-1}\times \g_{-1}\to \R$ is defined by $[U,V]=L(U,V)\psi$. The components $L:\g_{-1}^E\times \g_{-1}^E\to \R$ and $L:\g_{-1}^F\times \g_{-1}^F\to \R$ vanish. Thus $\flat$ provides isomorphisms $\g_{-1}^E\cong (\g_{-1}^F)^*\cong \g_1^F$ and $\g_{-1}^F\cong (\g_{-1}^E)^* \cong \g_1^E$. For $X\in \g_{-1}^E$ and $Y\in \g_{-1}^{F}$
\[
L(X,Y)=-X\cdot Y=\langle -X^T, Y \rangle
\]
so $X^{\flat}=-X^T\in \g_1^F$. Then for $y\in \g_1^F$, $y^{\sharp}=-y^T\in \g_{-1}^E$. For $Y\in \g_{-1}^F$ and $X\in \g_{-1}^E$
\[
L(Y,X)=Y\cdot X=\langle Y^T,X \rangle,
\]
so $Y^{\flat}=Y^T\in \g_1^E$. For $(a,A,b)\in \g_0$ and $X\in \g_{-1}^E$,
\begin{align*}
[(a,A,b),X]
&=(A-a)X.
\end{align*}
Next, we consider the bracket $\g_{-1}\times \g_1\to \g_0$. The components $\g_{-1}^E\times \g_1^F\to \g_0$ and $\g_{-1}^F\times\g_1^{E}\to \g_0$ are trivial. On the other hand, $[X,x]=(-x(X),Xx,0)$ and $[Y,y]=(0,-y^TY^T,y(Y))$. For $Z\in \g_{-1}^E$, 
\[
[[X,x],Z]=Xx(Z)+Zx(X)
\]
and
\[
[[Y,y],Z]=-y^TY^TZ=y^{\sharp}Y^{\flat}(Z).
\]

 As explained in \cite{CapSlovak2}, $\ker \square \subset C^1(\g_-,\g)$ consists of two $\g_0$-irreducible components, one taking $\g_{-1}^E\to \g_{-1}^F$ and one taking $\g_{-1}^{F}\to \g_{-1}^E$.

\subsection{Legendrian Contact Structures}
A \emph{Legendrian contact} structure is a $2n+1$-dimensional manifold $M$ with a contact distribution $H$ and two Legendrian distributions $E$ and $F$ for $H$ such that $H=E\oplus F$. In other words, $E$ and $F$ are $n$-dimensional and $[E,E]\subset H$ and $[F,F]\subset H$. An \emph{integrable Legendrian contact structure} is a Legendrian contact structure where, rather than assuming just that $E$ and $F$ are Legendrian, we assume the stronger condition that $E$ and $F$ are integrable, so $[E,E]\subset E$ and $[F,F]\subset F$. A Legendrian contact structure is sometimes called a Lagrangian contact structure or a paracontact structure, and an integrable Legendrian contact structure is sometimes called a bi-Legendrian foliation or a para-CR structure.

The basic approach taken in Cartan geometry is to view a family of geometric structures as curved or deformed versions of a homogeneous model space. The model space in the setting of Legendrian contact structures is the following.
\begin{example}[Model Space]
Define $\mathrm{Flag}_{1,n+1}(\R^{n+2})$ to be the space of $(1,n+1)$ flags $V_1\subset V_{n+1}\subset \R^{n+2}$. There are maps $\mathrm{Flag}_{1,{n+1}}(\R^{n+2})\to \mathrm{Gr}_1(\R^{n+2})$ and $\mathrm{Flag}_{1,{n+1}}(\R^{n+2})\to \mathrm{Gr}_{n+1}(\R^{n+2})$. Level sets of these maps define two foliations of dimension $n$. Let $E$ and $F$ be the distributions tangent to these foliations. The direct sum $E\oplus F$ is a contact distribution, and $E$ and $F$ are clearly integrable. The group $PGL(\R^{n+2})$ acts transitively on $\mathrm{Flag}_{1,n+1}(\R^{n+2})$ in a way that preserves $E$ and $F$.
\end{example}

A Legendrian contact structure may be described as a regular normal Cartan geometry modeled on $\mathrm{Flag}_{1,n+1}(\R^{n+2})$ under the action of the group $PGL_{n+2}(\R)$. As a Klein pair, this corresponds to $(PGL_{n+2}(\R),P)$ where $P$ is the parabolic subgroup stabilizing a standard flag $\langle e_1 \rangle \subset \langle e_1,\dots, e_{n+1}\rangle\subset \R^{n+2}$. Equivalently, $P$ is the set of equivalence classes of block upper triangular matrices with block sizes $1$, $n$ and $1$. Cartan geometrically, an integrable Legendrian contact structure is a Legendrian contact structure which is also torsion-free. The Lie algebras of $G$ and $P$ are the $\g$ and $\p$ considered in the previous section. The bundles $E$ and $F$ may be described as $\G_0\times_{G_0}\g_{-1}^E$ and $\G_0\times_{G_0}\g_{-1}^F$.

\begin{example}
 If $(M,[\nabla])$ is a projective manifold, the projectivized cotangent bundle $PT^*M$ inherits a Legendrian contact structure \cite{AsanoInoueSeo}. If $M$ is a surface, the projectivized cotangent bundle is $3$-dimensional and thus automatically integrable. However, if $M$ has dimension at least $3$, the projectivized cotangent bundle has dimension at least $5$ and its canonical Legendrian contact structure is integrable exactly when it is flat, which happens exactly when $(M,[\nabla])$ is projectively flat \cite{CapSlovak2}. 
\end{example}

We use Greek letter indices like $\alpha,\beta$ to refer to basis vectors in $E$, and Greek letter indices with bars like $\overline{\alpha},\overline{\beta}$ to refer to basis vectors in $F$.

\subsection{Tanaka--Webster Connection}
This section seeks to give an elementary definition of the Tanaka--Webster connection in Legendrian contact geometry. The uniqueness proof parallels the explicit construction of the Tanaka--Webster connection in \cite{CapSlovak2}. Given a contact form $\theta$, recall that $r$ is the Reeb field, the kernel of the projection $\pi:TM\to H$ is spanned by $r$, and $L:H\times H\to \R$ is defined by $L(X,Y)=\theta([X,Y])$. Specializing our definition of the Tanaka--Webster connection to integrable Legendrian contact geometry gives the following.
\begin{proposition}[Tanaka--Webster Connection in Legendrian Contact Geometry]
Given an integrable Legendrian contact manifold with a contact form $\theta$, there exists a unique affine connection $\widehat{\nabla}$ such that
\begin{enumerate}[(1)]
    \item $\widehat{\nabla}$ preserves $E$ and $F$,
    \item $\widehat{\nabla} r=0$,
    \item $\widehat{\nabla} L=0$,
    \item for $X,Y\in H$, $\widehat{\nabla}_X Y-\widehat{\nabla}_Y X-\pi([X,Y])=0$,
\item 
 $\widehat{\tau}(E)\subset F$ and $\widehat{\tau}(F)\subset E$.
\end{enumerate}
\end{proposition}
\begin{proof}
Let $\widehat{\nabla}$ be the Tanaka--Webster connection associated to a parabolic contact geometry. Because it is induced by a principal connection on $\G_0$, it preserves $E$ and $F$ and $\langle r \rangle$. Since it preserves $\theta$, it must preserve $r$. The fourth condition holds because an integrable Legendrian contact structure is torsion-free and this condition says $\widehat{T}^{(1)}|_{H\times H}=0$ and the fifth condition corresponds to the condition $\widehat{\tau}\in (\ker \square)_0\subset \End_0(H)$. This establishes existence.

Now we prove uniqueness. Since $\widehat{\nabla}r=0$, it suffices to consider $\widehat{\nabla}X$ for $X\in H$. First we prove uniqueness of the connection in contact directions. Suppose $X$ is a vector field in $E$ and $Y$ is a vector field in $F$. Then since $\widehat{\nabla}_XY-\widehat{\nabla}_YX=\pi([X,Y])$ it follows that $\widehat{\nabla}_X Y$ must be the $F$ component of $\pi([X,Y])$ and $\widehat{\nabla}_Y X$ must be the $E$ component of $\pi([X,Y])$. Now suppose without loss of generality that $X$ and $Y$ are both vector fields in $E$. Then for $Z$ a vector field in $F$, 
\[
L(\widehat{\nabla}_XY,Z)
=XL(Y,Z)-L(Y,\widehat{\nabla}_X Z).
\]
Since $\widehat{\nabla}_X Z$ is determined and $L$ is nondegenerate, this determines $\widehat{\nabla}_XY$.

Next we prove uniqueness in the Reeb direction. Suppose without loss of generality that $X\in E$. Then
\[
\widehat{\nabla}_rX=[r,X]+\widehat{\tau}(X).
\]
Since $\widehat{\tau}(X)\in F$, $\widehat{\nabla}_rX$ must be the $E$ component of $[r,X]$.
\end{proof}

\subsection{Modified Ricci Tensor}
Now we determine the modified Ricci tensor from the $A_l$ case of Theorem~\ref{thm::rhoTensor}. Recall that $\widehat{R}_{ijk}^{\ \ \ l}$ is the curvature tensor of the Tanaka--Webster connection and $\widehat{\Ric}_i^{\ j}=\frac{1}{2}\widehat{R}_{k \ i}^{\ k \ j}$. This tensor preserves the distributions $E$ and $F$. The complexification $H^{\C}$ of the contact distribution decomposes as $H^{\C}=E^{\C}\oplus F^{\C}$. The distributions $E^{\C}$ and $F^{\C}$ are in fact $\E$ and $\E$ from the setup to the theorem in the $A_l$ case. 

The complex trace of $\widehat{\Ric}$ acting on $E^{\C}$ is exactly the real trace of $\widehat{\Ric}$ acting on $E$. Therefore we may define $\widehat{R}$ as the trace of $\widehat{\Ric}$ acting on $E$. Furthermore, the endomorphism $\iota$ acting by $+1$ on $E^{\C}$ and $-1$ on $F^{\C}$ restricts to an endomorphism acting by $+1$ on $E$ and $-1$ on $F$. Then by Theorem~\ref{thm::rhoTensor} the modified Ricci tensor $\rho:H\to H$ is given by 
\[
\rho(X)=\frac{1}{n+2}\left(\widehat{\Ric}(X)-\frac{\widehat{R}}{2(n+1)}\iota(X) \right).
\]
\subsection{Harmonic Components}
The purpose of this subsection is to describe the harmonic components in Legendrian contact geometry. While this is also done in \cite{CapSlovak2}, we require slightly more information about the harmonic component in higher dimensions.

We will work with the complexified Lie algebra $\g^{\C}=\sl_{n+2}$ and grading, using $H^*_{\R}(\g_-,\g)^{\C}=H^*_{\C}(\g_-^{\C},\g^{\C})$ and the fact that Kostant's Borel--Weil--Bott computes the $\g_0^{\C}$-irreducible components in $H^*_{\C}(\g_-^{\C},\g^{\C})$, as explained in \cite{CapSlovak2}. Denote the reflection in $\h^*$ along the hyperplane orthogonal to $\alpha_i$ by $s_i$. Since $\g$ is split, the lowest weight elements we identify in the harmonic components will be contained in the real form $H^*_{\R}(\g_-,\g)$, and thus it does not make much of a difference whether we work with $\g$ or $\g^{\C}$. 

Recall that $I$ denotes the set of simple roots inducing the grading. In this section $\flat$ denotes the isomorphism $\p_+\cong \g_-^*$ induced by the fixed $\g$-invariant bilinear form, and $\eta_{\alpha}$ for $\alpha\in \Delta$ denotes a nonzero element of the root space $\g_\alpha$. If $s_is_j\in W^I$ is a length $2$ element of the associated Hasse diagram, Kostant's Borel--Weil--Bott implies there is a corresponding lowest weight representative of a $\g_0$-irreducible component in $H^2(\g_-,\g)=\ker \square\subset C^2(\g_-,\g)$ given by
\[
\eta_{\alpha_i}^{\flat}\wedge \eta_{s_i(\alpha_j)}^{\flat}\otimes \eta_{-s_is_j(\mu)}.
\]
Now assume that $\g=\sl_{n+2}$ with the standard contact grading. The analysis of harmonic curvature components is significantly different when $n=1$, corresponding to $3$-dimensional Legendrian contact geometry, and when $n>1$, corresponding to higher-dimensional Legendrian contact geometries. If $n=1$, then $\g=\sl_3$, which has rank $2$. The set of simple roots inducing the contact grading is $I=\{\alpha_1,\alpha_2\}$, and the length $2$ Hasse diagram elements are $W^{I}(2)=\{s_1s_2,s_2s_1\}$. Fix the convention that for $\alpha$ a root, $\eta_{\alpha}$ is a nonzero element of $\g_\alpha$. The lowest weight representatives of the two harmonic components in $H^2(\g_-,\g)$ are 
\[
\eta_{\alpha_1}^{\flat}\wedge \eta_{s_1(\alpha_2)}^{\flat} \otimes \eta_{-s_1s_2(\mu)}
=\eta_{\alpha_1}^{\flat}\wedge \eta_{\alpha_1+\alpha_2}^{\flat} \otimes \eta_{\alpha_1}
\]
and swapping the indices $1$ and $2$,
\[
\eta_{\alpha_2}^{\flat}\wedge \eta_{\alpha_1+\alpha_2}^{\flat} \otimes \eta_{\alpha_2}.
\]
These representatives correspond to maps $Q\times E\to E^*$ and $Q\times F\to F^*$. These harmonic components will give us the tensors $Q_{11}$ and $Q_{\1\1}$.

If $n>1$, then $\g=\sl_{n+2}$ has rank larger than $2$, the chosen simple roots are $I=\{\alpha_1,\alpha_{n+1}\}$, and $W^I(2)=\{s_1s_2,s_{n+1}s_n,s_1s_{n+1}\}$. The lowest weight representatives of the three harmonic components in $H^2(\g_-,\g)$ are 
\[
\eta_{\alpha_1}^{\flat}\wedge \eta_{s_1(\alpha_2)}^{\flat} \otimes \eta_{-s_1s_2(\mu)}
=\eta_{\alpha_1}^{\flat}\wedge \eta_{\alpha_1+\alpha_2}^{\flat} \otimes \eta_{\alpha_1-\mu}
\]
and applying the nontrivial Dynkin diagram automorphism,
\[
\eta_{\alpha_{n+1}}^{\flat}\wedge \eta_{\alpha_n+\alpha_{n+1}}^{\flat} \otimes \eta_{\alpha_{n+1}-\mu},
\]
and finally
\[
\eta_{\alpha_1}^{\flat}\wedge \eta_{s_1(\alpha_{n+1})}^{\flat}\otimes \eta_{-s_1s_{n+1}(\mu)}
=\eta_{\alpha_1}^{\flat}\wedge \eta_{\alpha_{n+1}}^{\flat}\otimes \eta_{\alpha_1+\alpha_{n+1}-\mu}.
\]
These three representatives correspond to maps $E\times E\to F$, $F\times F\to E$ and $E\times F\to \End_0(H)$. The first two are torsions, and the third will correspond to our tensor $C_{\alpha\overline{\beta}\gamma\overline{\sigma}}$. The main reason we have redone these computations instead of relying on \cite{CapSlovak2} is to point out that $\eta_{\alpha_1+\alpha_{n+1}-\mu}\in \g_0^{ss}$, and therefore acts trivially on $\g_{-2}$. The corresponding harmonic component $E\times F\to \End_0(H)$ has output acting trivially on $Q$, and therefore acts on $H$ by transformations preserving $L$.
\subsection{Harmonic Curvature in Dimension $3$}
As discussed above, harmonic components in dimension $3$ correspond to maps $Q\times E\to E^*$ and $Q\times F\to F^*$. The distributions $E$ and $F$ are $1$-dimensional, with basis vectors $X_1$ and $X_{\1}$. Observe that $\rho_1^{\ 1}=\frac{\widehat{R}}{4}$ and $\rho_{\overline{1}}^{\ \overline{1}}=-\frac{\widehat{R}}{4}$.
Since $\P^{(4)}(r)$ corresponds to a multiple of $\phi$,
\[
\partial \P^{(4)}_{011}
=-\{X_1,\P^{(4)}(r)\}(X_1)
\]
is a multiple of 
\[
\{X_1,\phi\}(X_1)=-X_1^{\flat}(X_1)=0.
\]
Since $\tau_{11}=\widehat{\tau}_{11}$,
\begin{align*}
\nabla_0 \tau_{11}
&=\widehat{\nabla}_0 \tau_{11}
+L(\tau(\rho(X_1)),X_1)+L(\tau(X_1),\rho(X_1))\\
&=\widehat{\nabla}_0\tau_{11}+\frac{\widehat{R}}{4}(\tau_{11}+\tau_{11})\\
&=\widehat{\nabla}_0 \widehat{\tau}_{11}+\frac{\widehat{R}}{2}\widehat{\tau}_{11}.
\end{align*}
Since $(\widehat{\tau}\circ\widehat{\tau})_{11}=(\rho\circ \rho)_{11}=0$,
\begin{align*}
(\tau\circ \tau)_{11}
&=-(\rho\circ \widehat{\tau}+\widehat{\tau}\circ \rho)_{11}\\
&=\frac{\widehat{R}}{4}\widehat{\tau}_{11}
-\frac{\widehat{R}}{4}\widehat{\tau}_{11}\\
&=0.
\end{align*}
Since $\nabla_1=\widehat{\nabla}_1$ and $\nabla_{\1}=\widehat{\nabla}_{\1}$,
\begin{align*}
\kappa_{011}
&=Y_{011}+\partial \P_{011}\\
&=Y_{011}\\
&=-\nabla_0 \tau_{11}
+\frac{2}{3}
\nabla_1\nabla_k\tau_{1}^{\ k}-(\tau\circ \tau)_{11}\\
&=-\widehat{\nabla}_0 \widehat{\tau}_{11}
-\frac{2}{3}
\widehat{\nabla}_1\widehat{\nabla}_{1}\rho_{1}^{\ 1}
+\frac{2}{3}
\widehat{\nabla}_1\widehat{\nabla}_{\1}\widehat{\tau}_{1}^{\ \overline{1}}
-\frac{\widehat{R}}{2}\widehat{\tau}_{11}\\
&=-\widehat{\nabla}_0 \widehat{\tau}_{11}
-\frac{1}{6}
\widehat{\nabla}_1\widehat{\nabla}_1\widehat{R}
+\frac{2}{3}
\widehat{\nabla}_1\widehat{\nabla}_{\overline{1}}\widehat{\tau}_1^{\1}
-\frac{\widehat{R}}{2}
\widehat{\tau}_{11}.
\end{align*}
Define 
\begin{equation}\label{eqn::Umbilical1}
    Q_{11}=\kappa_{011}=-\widehat{\nabla}_0 \widehat{\tau}_{11}
-\frac{1}{6}
\widehat{\nabla}_1\widehat{\nabla}_1\widehat{R}
+\frac{2}{3}
\widehat{\nabla}_1\widehat{\nabla}_{\overline{1}}\widehat{\tau}_1^{\1}
-\frac{\widehat{R}}{2}
\widehat{\tau}_{11}.
\end{equation}
Similarly, define
\begin{equation}\label{eqn::Umbilical2}
    Q_{\1\1}=\kappa_{0\1\1}=-\widehat{\nabla}_0 \widehat{\tau}_{\1 \1}
+\frac{1}{6}
\widehat{\nabla}_{\1}\widehat{\nabla}_{\1}\widehat{R}
+\frac{2}{3}
\widehat{\nabla}_{\1}\widehat{\nabla}_{1}\widehat{\tau}_{\1}^{\ 1}
+\frac{\widehat{R}}{2}
\widehat{\tau}_{\1 \1}.
\end{equation}
These are the two harmonic curvature components. A regular normal parabolic Cartan geometry is flat if and only if its harmonic curvature vanishes. Part (a) of Theorem~\ref{thm::HarmonicCurvature} now follows.
\subsection{Harmonic Curvature in Higher Dimensions}
Since an integrable Legendrian contact structure is torsion-free, the $E\times E\to F$ and $F\times F\to E$ components of harmonic curvature vanish. The remaining component consists of maps $E\times F\to \End_0(H)$ whose values in $\End_0(H)$ act on $H$ in a way that preserves $L$. Let $X\in E$ and $Y\in F$. Then
\[
\kappa_0(X,Y)
=R(X,Y)+\partial \P^{(2)}(X,Y).
\]
It is a standard result that the nonzero curvature component of lowest homogeneity is automatically harmonic. Since curvature vanishes in homogeneity $1$, the component $\kappa_0$ in homogeneity $2$ is harmonic. Therefore it is valued in the part of $\End_0(H)$ that preserves $L$. Since $L$ induces an isomorphism $F\cong E^*$ and $\kappa_0(X,Y)$ preserves $E$ and $F$, the action of $\kappa_0(X,Y)$ on $F$ can be recovered from its action on $E$. Then, to show vanishing of $\kappa_0(X,Y)\in \End_0(H)$, it suffices to show that it acts trivially on $E$. 
Observe that
\[
\nabla_{[X,Y]}Z
=\nabla_{\pi([X,Y])}Z
+\theta([X,Y])\nabla_r Z
=\widehat{\nabla}_{[X,Y]}Z-L(X,Y)\rho(Z).
\] 
Therefore
\begin{align*}
R(X,Y)Z
&=\nabla_X \nabla_Y Z -\nabla_Y \nabla_XZ-\nabla_{[X,Y]}Z\\
&=\widehat{\nabla}_X \widehat{\nabla}_Y Z 
-\widehat{\nabla}_Y \widehat{\nabla}_X Z
-\widehat{\nabla}_{[X,Y]}Z
+L(X,Y)\rho(Z)\\
&=\widehat{R}(X,Y)Z+L(X,Y)\rho(Z).
\end{align*}
For the other term,
\[
\partial \P^{(2)}(X,Y)
=\{X,\P^{(2)}(Y)\}-\{Y,\P^{(2)}(X)\}
=\{X,(\rho-\widehat{\tau})(Y)^{\flat}\}
-\{Y,(\rho-\widehat{\tau})(X)^{\flat}\}.
\]
Since $X\in E$ and $\widehat{\tau}(Y)^{\flat}\in F^*$,
\[
\{X,\widehat{\tau}(Y)^{\flat}\}=0.
\]
Similarly,
\[
\{Y,\widehat{\tau}(X)^{\flat}\}=0.
\]
Let $Z\in E$. Then
\[
\{X,\rho(Y)^{\flat}\}(Z)
=X\rho(Y)^{\flat}(Z)
+Z\rho(Y)^{\flat}(X)
\]
and
\[
\{Y, \rho(X)^{\flat}\}(Z)
=\rho(X)Y^{\flat}(Z).
\]
Thus
\begin{align*}
\partial (\P^{(2)})_{\alpha \overline{\beta} \gamma}^{\ \ \ \ \sigma}
&=\delta_{\alpha}^{\ \sigma}\rho_{\overline{\beta}\gamma}
+\delta_{\gamma}^{\ \sigma}\rho_{\overline{\beta}\alpha}
-L_{\overline{\beta}\gamma}\rho_{\alpha}^{\ \sigma}.
\end{align*}
Putting it all together,
\[
\kappa_{\alpha \overline{\beta}\gamma}^{\ \ \ \ \sigma}
=\widehat{R}_{\alpha \overline{\beta} \gamma}^{\ \ \ \ \sigma}+L_{\alpha \overline{\beta}}\rho_{\gamma}^{\ \sigma}
+\delta_{\alpha}^{\ \sigma}\rho_{\overline{\beta}\gamma}
+\delta_{\gamma}^{\ \sigma}\rho_{\overline{\beta}\alpha}
+L_{\gamma\overline{\beta}}\rho_{\alpha}^{\ \sigma}.
\]
This tensor vanishes if and only if the harmonic curvature component vanishes, which happens if and only if the corresponding Cartan geometry is flat.
Lowering an index, define
\begin{equation}\label{eqn::ChernMoser}
C_{\alpha\overline{\beta}\gamma\overline{\sigma}}=\kappa_{\alpha \overline{\beta}\gamma \overline{\sigma}}
=\widehat{R}_{\alpha \overline{\beta} \gamma \overline{\sigma}}
+L_{\alpha \overline{\beta}}\rho_{\gamma \overline{\sigma}}
+L_{\alpha \overline{\sigma}}\rho_{\gamma\overline{\beta}}
+L_{\gamma \overline{\sigma}}\rho_{\alpha\overline{\beta}}
+L_{\gamma \overline{\beta}}\rho_{\alpha \overline{\sigma}}.
\end{equation}
Part (b) of Theorem~\ref{thm::HarmonicCurvature} now follows.

\appendix
\section{Kostant Laplacian Eigenvalues}
\subsection{General Formula}
The Kostant Laplacian $\square$ is unchanged if we work with the complexification of $\g$, so we will assume $\g$ is complex for the remainder of this subsection. Let $\delta$ be the \emph{Weyl vector}, defined by
\[
\delta=\sum \lambda_i=\frac{1}{2}\sum_{\alpha\in \Delta^+} \alpha.
\]
The Kostant Laplacian $\square$ acts by a scalar on $\g_0$-irreducible subrepresentations in $C^*(\g_-,\g)$. We can quantify these scalars precisely. Let $V^{\nu}\subset C^*(\g_-,\g)$ be a $\g_0$-irreducible representation whose dual has highest $\g_0$ weight $\nu$. Recall that $\mu$ is the highest root of $\g$. Then $\square$ multiplies the elements of $V^{\nu}$ by the scalar
\begin{equation}\label{eqn::KostantScalar}
\frac{1}{2}(|\nu|^2-|\mu|^2)+\langle \nu-\mu, \delta\rangle.
\end{equation}
\subsection{For $1$-cochains}\label{app:1Cochains}
The Kostant Laplacian acts by a scalar on each $\g_0$-irreducible component in $C^*(\g_-,\g)$, according to equation \eqref{eqn::KostantScalar}. This section computes this scalar for several isotypic components or submodules, and these results are later used to compute the Rho tensor according to Proposition~\ref{prop::RhoTensor}. These constants admit a uniform description in terms of the dimension of the parabolic contact manifold and otherwise do not depend on the choice of simple Lie algebra $\g$. Recall that $n=\frac{1}{2}\dim(\g_{-1})$.
\begin{proposition}
$\langle \mu, \delta \rangle=n+1$.
\end{proposition}
\begin{proof}
    If $\alpha$ is a root, $\langle \mu, \alpha\rangle=\alpha(E)$ which is either $0,1$ or $2$ depending on the homogeneity in which $\g_{\alpha}$ appears. Then
    \[
    \langle \mu, \delta\rangle
    =\frac{1}{2}\sum_{\alpha\in \Delta^+}
    \langle \mu, \alpha\rangle
    =\frac{1}{2}(\dim(\g_1)+2)
    =n+1.
    \]
\end{proof}

\begin{proposition}\label{prop::KostantEigenvalues}
\begin{enumerate}[(a)]
\item On $\g_{-1}^*\otimes \g_2$ the Kostant Laplacian $\square$ acts by $-(2n+1)$. 
\item On $\g_{-2}^*\otimes \g_1$ the Kostant Laplacian $\square$ acts by $-(2n+1)$.
\item On $\g_{-2}^*\otimes \g_2$ the Kostant Laplacian $\square$ acts by $-3n$.
\end{enumerate}
\end{proposition}
\begin{proof}
    \begin{enumerate}[(a)]
    \item The dual is $(\g_{-1}^*\otimes \g_2)^*\cong \g_{-1}\otimes \g_{-2}$. A lowest weight in $\g_1$ is a simple root $\alpha_k$ such that $\langle \alpha_k, \mu \rangle=1$. Then a highest weight in $\g_{-1}$ is $-\alpha_k$. A highest weight in $\g_{-1}\otimes \g_{-2}$ is $-(\alpha_k+\mu)$. We have
\[
|-(\alpha_k+\mu)|^2-|\mu|^2
=|\alpha_k|^2+2\langle \alpha_k, \mu \rangle
=|\alpha_k|^2+2
\]
and
\[
\langle -\alpha_k-2\mu, \delta \rangle
=-\langle \alpha_k, \delta\rangle
-2(n+1)
=-\frac{1}{2}|\alpha_k|^2-2(n+1).
\]
Using equation \eqref{eqn::KostantScalar}, the action of $\square$ is given by
\[
\frac{1}{2}(|\alpha_k|^2+2)-\frac{1}{2}|\alpha_k|^2-2(n+1)
=-(2n+1).
\]
    \item Since $\g_{-2}^*\otimes \g_1 \cong \g_{-1}^*\otimes \g_2$, this subrepresentation is in the same isotypic component and the action of $\square$ is the same.
    \item  We have $(\g_{-2}^*\otimes \g_2)^*\cong \g_{-2}\otimes \g_{-2}$, which has weight $\nu=-2\mu$. Then
    \[
    |-2\mu|^2-|\mu|^2=3|\mu|^2=6
    \]
    while
    \[
    \langle -2\mu-\mu, \delta\rangle
    =-3\langle\mu, \delta \rangle
    =-3(n+1).
    \]
    Then by equation \eqref{eqn::KostantScalar}, $\square$ acts by
    \[
   \frac{1}{2}(6)-3(n+1)=-3n.
    \]
    \end{enumerate}
\end{proof}
\subsection{For $0$-cochains}\label{app::0Cochains}
Consistent with the rest of the paper, we assume the normalization condition $|\mu|^2=2$, where $\mu$ is the highest root. This is equivalent to $|E|^2=2$, where $E$ is the grading element relative to the corresponding contact gradings. Figure \ref{fig::DynkinDiagrams} shows the Dynkin diagrams with the value of $|\alpha_i|^2$ written next to the node $\alpha_i$.

\begin{lemma}\label{lemma::NormSame}
Let $\g$ be a contact-graded complex simple Lie algebra. Let $\s\subset \g_0$ be a simple ideal of highest $\g_0$ weight $\nu$. If $\g$ is not isomorphic to $B_3$ or $G_2$ then $|\nu|^2=2$.
\end{lemma}
\begin{proof}
    Proposition 3.2.2 of \cite{CapSlovak2} says that the Dynkin diagram of $\g_0^{ss}$ is obtained by removing the crossed nodes in the Dynkin diagram of $\g$. As a result, the Dynkin diagram of $\s$ is a connected component of the remaining diagram. Let $S\subset \Delta^0$ be the simple roots on this component. Let $\h'$ be the Cartan subalgebra of $\s$. For $\alpha_i\in S$, the restriction $\alpha'_i=\alpha_i|_{\h'}$ is a simple root of $\s$. The roots of $\g$ whose root spaces are contained in $\s$ are those contained in the span of $S$. Then $\nu$ is contained in the span of $S$. Furthermore, $\nu|_{\h'}=\mu_{\s}$ is the highest root of $\s$. These two conditions uniquely characterize $\nu$ as $\sum c_i \alpha_i$ where $\sum c_i \alpha'_i=\mu_{\s}$. 
    
    Because we know that the Dynkin diagram of $\s$ is the Dynkin diagram of $\g$ restricted to $S$, the Cartan integers between the simple roots $\alpha'_i$ of $\s$ are the same as those between the corresponding $\alpha_i\in S$. Since $\g$ and $\s$ are simple Lie algebras, a choice of $\g$-invariant or $\s$-invariant bilinear form on $\g$ or $\s$ is unique up to scale, so the equality of Cartan integers
    \[
    \frac{2\langle \alpha'_i,\alpha'_j\rangle}{|\alpha'_i|^2}=\frac{2\langle \alpha_i,\alpha_j\rangle}{|\alpha_i|^2} 
    \]
    is still true when we replace the Killing inner products by the unique invariant bilinear forms for which $|\mu|^2=2$ and $|\mu_{\s}|^2=2$. 

    By the existence of a simple ideal in $\g_0$, the Lie algebra $\g$ is not isomorphic to $A_2$. It follows from an inspection of Figure \ref{fig::DynkinDiagrams} that if $\g$ is not isomorphic to $B_3$ or $G_2$ then if we remove a crossed node and consider a connected component $S$, the pattern of values of $|\alpha_i|^2$ is the same as that in the corresponding Dynkin diagram. This means that for $\alpha_i\in S$ we get $|\alpha_i|^2=|\alpha'_i|^2$. Because the Cartan integers and simple root norms for $\alpha_i\in S$ relative to $\g$ are the same as those of $\alpha'_i$ relative to $\s$, it follows that
    \begin{align*}
    |\nu|^2
    &=|\sum c_i \alpha_i|^2\\
    &=|\sum c_i \alpha'_i|^2\\
    &=|\mu_{\s}|^2\\
    &=2.
    \end{align*}
\end{proof}
Using the formula
\[
n=\langle \mu, \delta \rangle-1,
\]
we compute the following well-known values for the exceptional simple Lie algebras. Recall that $n=\frac{1}{2}\dim(\g_{-1})$ relative to the unique contact grading on $\g$.
\begin{figure}
\includegraphics[scale=.14]{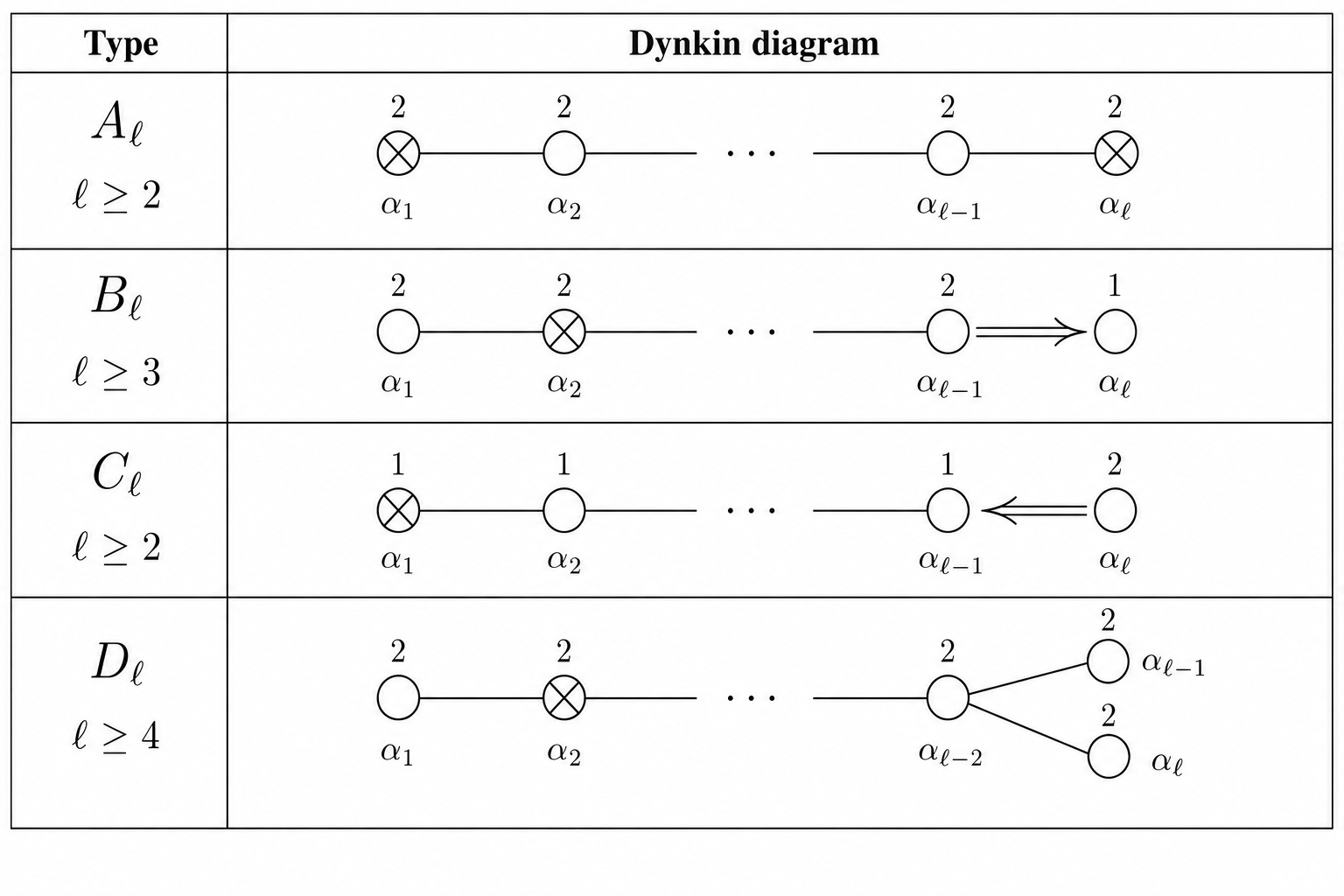}
\includegraphics[scale=.14]{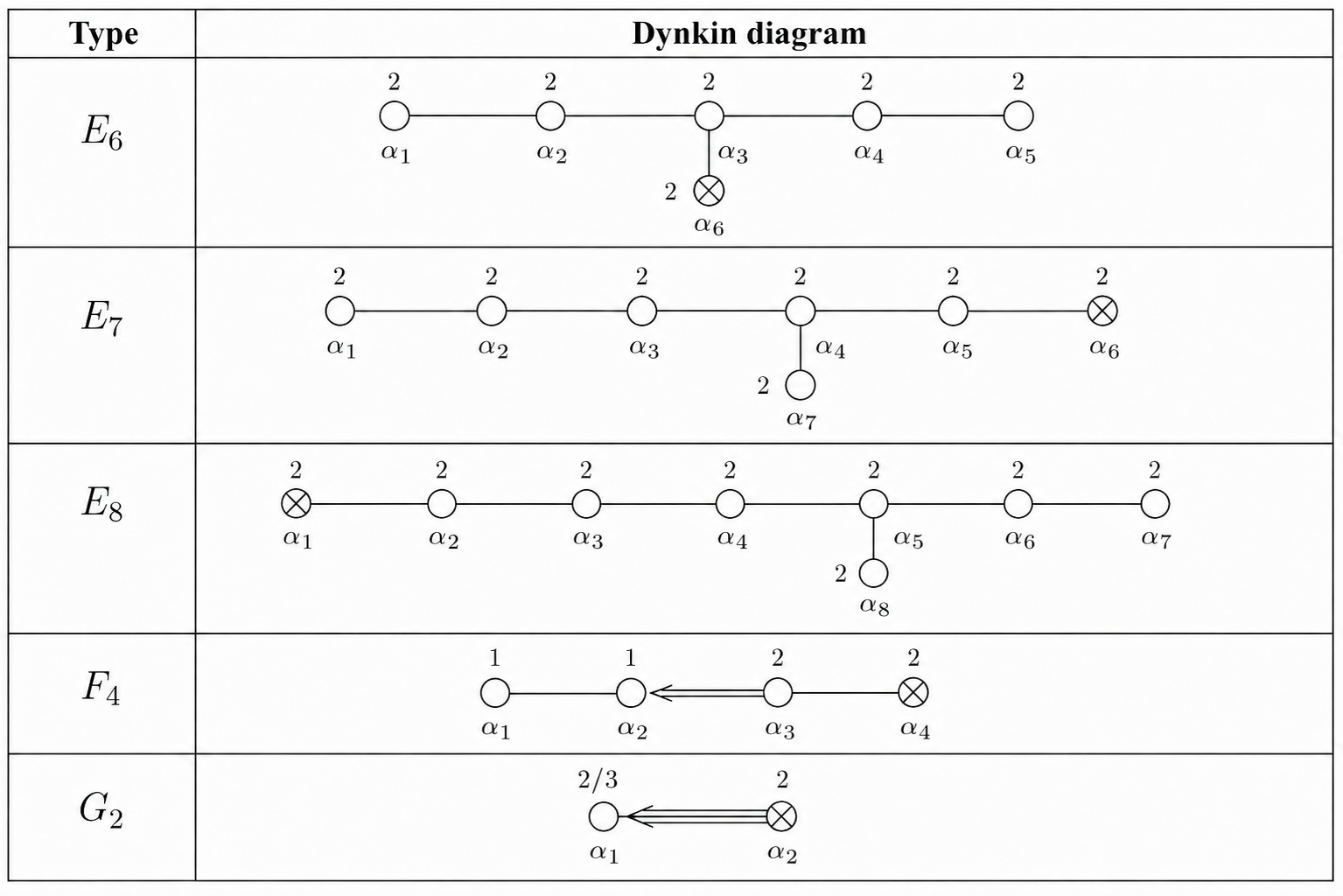}
\caption{Dynkin Diagrams}
\label{fig::DynkinDiagrams}
\end{figure}
\begin{proposition}
In the exceptional simple Lie algebras, $n$ takes the following values.
\begin{itemize}
    \item If $\g\cong E_6$, $n=10$, 
    \item if $\g\cong E_7$, $n=16$,
    \item if $\g\cong E_8$, $n=28$,
    \item if $\g\cong F_4$, $n=7$,
    \item if $\g\cong G_2$, $n=2$.
\end{itemize}
\end{proposition}

Recall that the action of $\square$ on a $\g_0$-irreducible subrepresentation $V^{\nu}\le C^*(\g_-,\g)$ whose dual has highest weight $\nu$ is given by
\begin{equation}\label{eqn::Eigenvalue2}
\square=\frac{1}{2}(|\nu|^2-|\mu|^2)+\langle \nu-\mu, \delta \rangle.
\end{equation}
In particular, suppose that $V^{\nu}\subset \g_0$. It is either a trivial representation or a simple ideal in $\g_0$. Either way, $V^{\nu}$ is self-dual and thus has highest weight $\nu$. It follows from Lemma~\ref{lemma::NormSame} that if $V^{\nu}$ is a simple ideal and $\g$ is not isomorphic to $B_3$ or $G_2$ then $|\nu|^2=|\mu|^2$ and equation \eqref{eqn::Eigenvalue2} reduces to \begin{equation}\label{eqn::Eigenvalue3}
    \square=\langle\nu-\mu, \delta \rangle.
\end{equation}
    For the following proof, it may also be useful to combine this with 
\begin{equation}\label{eqn::Eigenvalue4}
n-\square
=n+\langle \mu-\nu, \delta \rangle
=2n+1-\langle \nu, \delta \rangle.
\end{equation}
Furthermore, note that in the simply laced Dynkin diagrams of type $A,D$ and $E$ we have $|\alpha_i|^2=2$ for every simple root $\alpha_i\in \Delta^0$, so the simple roots $\alpha_i$ and the fundamental weights $\lambda_i$ form
dual bases, which slightly simplifies the calculation of some inner products.
\begin{lemma}\label{lemma::LaplacianEigenvalues}
Let $\g$ be a contact-graded complex simple Lie algebra.
\begin{enumerate}[(1)]
\item The operator $n-\square$ acts on $\z(\g_0)$ by $2(n+1)$.
\item If $\g\cong A_l$ for $l\ge 2$, then $n-\square$ acts on $\g_0^{ss}$ by $n+2$.
\item If $\g\cong B_l$ for $l\ge 3$ or $\g\cong D_l$ for $l\ge 4$, let $\g_0^{ss}=\g_0'\oplus \g_0''$ where $\g_0'$ corresponds to the simple root $\alpha_1$, and $\g_0''$ corresponds to simple roots $\alpha_i$ for $i\ge 3$. Then $n-\square$ acts on $\g_0'$ by $2n$ and on $\g_0''$ by $n+4$. 
\item If $\g\cong C_l$, for $l\ge 2$ then $n-\square$ acts on $\g_0^{ss}$ by $n+1$.
\item If $\g\cong E_6$, then $n-\square$ acts on $\g_0^{ss}$ by $16$.
\item If $\g\cong E_7$, then $n-\square$ acts on $\g_0^{ss}$ by $24$.
\item If $\g\cong E_8$, then $n-\square$ acts on $\g_0^{ss}$ by $40$.
\item If $\g\cong F_4$, then $n-\square$ acts on $\g_0^{ss}$ by $12$.
\item If $\g\cong G_2$, then $n-\square$ acts on $\g_0^{ss}$ by $\frac{16}{3}$.
\end{enumerate}
\end{lemma}
\begin{proof}
    \begin{enumerate}[(1)]
    \item An element of $\z(\g_0)$ belongs to a trivial $\g_0$-representation of highest weight $\nu=0$, so
\begin{align*}
\square
&=-\frac{1}{2}(|\mu|^2)-\langle\mu, \delta \rangle\\
&=-\frac{1}{2}(2)-(n+1)\\
&=-(n+2)
\end{align*}
and $n-\square$ acts by $2(n+1)$.

\item The algebra $\g_0^{ss}$ has highest weight $\nu=\mu-\alpha_1-\alpha_{l}$. As a result, $\square$ acts on $\g_0^{ss}$ by
\[
\langle \nu-\mu, \delta\rangle
=-\frac{|\alpha_1|^2}{2}
-\frac{|\alpha_l|^2}{2}
=-2
\]
and $n-\square$ acts by $n+2$.
\item We have $\g_0'\cong A_1$ and either $\g_0''\cong B_{l-2}$ or $\g_0''\cong D_{l-2}$. In particular, in the $D_l$ case for $l=4$, $D_{l-2}\cong A_1 \oplus A_1$. First consider the action of $\square$ on $\g_0'$. The highest weight is $\nu=\alpha_1$. Even if $\g\cong B_3$ we still have $|\alpha_1|^2=|\mu|^2=2$ so we can use equation \eqref{eqn::Eigenvalue3}. Then $\square$ acts by
\[
\langle \nu-\mu,\delta\rangle
=\frac{|\alpha_1|^2}{2}-(n+1)
=-n
\]
and $n-\square$ acts by $2n$.

Now consider the action of $\square$ on $\g_0''$. We divide into three cases: \textbf{(1)} $B_3$, \textbf{(2)} $D_4$, and \textbf{(3)} $B_l$ for $l\ge 4$ or $D_l$ for $l\ge 5$. 

Case \textbf{(1)}. In the $B_3$ case, $\square$ acts on $\g_0''$ by
\begin{align*}
\frac{1}{2}(|\alpha_3|^2-|\mu|^2)
+\langle \alpha_3-\mu, \delta \rangle
=-\frac{1}{2}+\frac{|\alpha_3|^2}{2}-4
=-4.
\end{align*}

Case \textbf{(2)}. In the $D_4$ case, the ideal $\g_0''$ has two simple ideals of rank $1$. Since they are related by Dynkin diagram automorphism preserving the crossed node $\alpha_2$, the action of $\square$ on either of those simple ideals is equal to its action on $\g_0'$. We computed above that this is equal to $-n$. However, since $n=\langle \delta, \mu \rangle-1=4$, the Kostant Laplacian $\square$ acts on $\g_0''$ by $-4$.

Case \textbf{(3)}.
Otherwise $\g$ has type $B_{l}$ for $l\ge 4$ or $D_{l}$ for $l\ge 5$. Either way, we use equation \eqref{eqn::Eigenvalue3}. We have $\nu=\mu-\alpha_1-2\alpha_2-\alpha_3$ so $\nu-\mu=-(\alpha_1+2\alpha_2+\alpha_3)$ and $\square$ acts by
\[
\langle -(\alpha_1+2\alpha_2+\alpha_3),\delta\rangle
=-\left(\frac{|\alpha_1|^2}{2}
+|\alpha_2|^2
+\frac{|\alpha_3|^2}{2} \right)=-4.
\]
Therefore in all three cases $\square$ acts on $\g_0''$ by $-4$ and $n-\square$ acts by $n+4$.

 \item In type $C_l$, $\g_0^{ss}$ has highest weight $\mu-2\alpha_1$. Then $\square$ acts on $\g_0^{ss}$ by
\[
\langle \nu-\mu, \delta \rangle
=\langle-2\alpha_1, \delta \rangle
=-|\alpha_1|^2
=-1.
\]
Then $n-\square=n+1$.

\item
 For $E_6$, $\g_0^{ss}$ has type $A_5$, $\nu=\alpha_1+\alpha_2+\alpha_3+\alpha_4+\alpha_5$ and $\langle \nu, \delta  \rangle=5$. Using equation \eqref{eqn::Eigenvalue4}, $n-\square=21-5=16$.

 \item
 For $E_7$, $\g_0^{ss}$ has type $D_6$, $\nu=\alpha_1+2\alpha_2+2\alpha_3+2\alpha_4+\alpha_5+\alpha_7$, and $\langle \nu, \delta \rangle=9$. Using equation \eqref{eqn::Eigenvalue4}, $n-\square=33-9=24$. 

 \item
 For $E_8$, $\g_0^{ss}$ has type $E_7$, $\nu=\alpha_2+2\alpha_3+3\alpha_4+4\alpha_5+3\alpha_6+2\alpha_7+2\alpha_8$ and $\langle\nu, \delta \rangle=17$. Using equation \eqref{eqn::Eigenvalue4}, $n-\square=57-17=40$.

\item In type $F_4$, $\g_0^{ss}$ has type $C_3$, $n=\langle \mu, \delta \rangle-1=7$, $2n+1=15$, and $\nu=2\alpha_1+2\alpha_2+\alpha_3$. Then 
\[
\langle \nu, \delta \rangle
=2\frac{|\alpha_1|^2}{2}
+2\frac{|\alpha_2|^2}{2}
+\frac{|\alpha_3|^2}{2}
=1+1+1=3.
\]
Using equation \eqref{eqn::Eigenvalue4}, $n-\square=15-3=12$.

\item In type $G_2$ we have $\nu=\alpha_1$, $|\nu|^2=\frac{2}{3}$, $\langle \nu, \delta\rangle=\frac{1}{3}$, $n=\langle \delta, \mu \rangle-1=2$ and $2n+1=5$. Using equation \eqref{eqn::Eigenvalue2}, $\square$ acts by
\[
\frac{1}{2}(|\nu|^2-|\mu|^2)
+\langle \nu-\mu, \delta \rangle
=-\frac{2}{3}+\langle \nu, \delta\rangle-(n+1)
=-\frac{1}{3}-(n+1)
\]
so
\[
n-\square=2n+1+\frac{1}{3}=\frac{16}{3}.
\]
    \end{enumerate}
\end{proof}


\begin{thebibliography}{99}
\bibitem{Alt}
Alt, J. (2011). \emph{Weyl connections and the local sphere theorem for quaternionic contact structures}. Annals of Global Analysis and Geometry, Vol 39, 165--186.

\bibitem{AsanoInoueSeo}
Asano, K., Inoue, K., Seo, W. (1994). \emph{Lagrangian contact structures on projective cotangent bundles}. Osaka Journal of Mathematics, Vol 31(4), 837--860.

\bibitem{BryantGriffithsHsu}
Bryant, R., Griffiths, P., Hsu, L. (1995). \emph{Toward a geometry of differential equations}. Geometry, Topology and Physics, Vol 4, 1--76.

\bibitem{Cap}
\v Cap, A. (2008). \emph{Infinitesimal automorphisms and deformations of parabolic geometries}. Journal of the European Mathematical Society, Vol 10(2), 415--437.

\bibitem{CapGuo}
\v Cap, A., Guo, Z. (2026). \emph{Weyl structures for path geometries}. arXiv:2604.12846.

\bibitem{CapSlovak1}
\v Cap, A., Slov\'ak, J. (2003). \emph{Weyl structures for parabolic geometries}. Mathematica Scandinavica, Vol 93(1), 53--90.

\bibitem{CapSlovak2}
\v Cap, A., Slov\'ak, J. (2009). \emph{Parabolic Geometries I: Background and General Theory}. Mathematical Surveys and Monographs, Vol 154, American Mathematical Society.

\bibitem{ChengLee}
Cheng, J., Lee, J. (1990). \emph{The Burns-Epstein invariant and deformation of CR structures}. Duke Mathematical Journal, Vol 60(1), 221--254.

\bibitem{ChernMoser}
Chern, S.S., Moser, J.K. (1974). \emph{Real hypersurfaces in complex manifolds}. Acta Mathematica, Vol 133, 219--271. 

\bibitem{DragomirTomassini} 
Dragomir, S., Tomassini, G. (2006). \emph{Differential Geometry and Analysis on CR Manifolds}. Progress in Mathematics, Vol 246, Birkh\"auser.

\bibitem{Fox}
Fox, D.J.F. (2005). \emph{Contact Projective Structures}. Indiana University Mathematics Journal, Vol 54(6), 1547--1598.

 \bibitem{GoverGraham}
Gover, A.R., Graham, C.R. (2005). \emph{CR invariant powers of the sub-Laplacian}. Journal f{\"u}r die reine und angewandte Mathematik, Vol 583, 1--27.

\bibitem{Guo}
Guo, Z. (2025). \emph{Two Fefferman-type constructions involving almost Grassmann structures and path geometries}. arXiv:2509.04878.

\bibitem{Harrison} Harrison, J. (1995). \emph{Some problems in the invariant theory of parabolic geometries}, PhD thesis, University of Edinburgh.

\bibitem{Herzlich}
Herzlich, M. (2009). \emph{The canonical Cartan bundle and connection in CR geometry}. Mathematical Proceedings of the Cambridge Philosophical Society, Vol 146, 415--434.

\bibitem{IvanovVassilevZamkovoy}
Ivanov, S., Vassilev, D., Zamkovoy, S. (2010). \emph{Conformal paracontact curvature and the local flatness theorem}. Geometriae Dedicata, Vol 144, 79--100.

\bibitem{KobayashiNomizu}
Kobayashi, S., Nomizu, K. (1963). \emph{Foundations of Differential Geometry: Volume 1}. Interscience Publishers, reprinted by Wiley Classics Library. 
 
 \bibitem{Lee} Lee, J. (1988). \emph{Pseudo-Einstein structures on CR manifolds}. American Journal of Mathematics, Vol 110(1), 157--178.

 \bibitem{Matsumoto} Matsumoto, Y. (2022). \emph{The CR Killing operator and Bernstein-Gelfand-Gelfand construction in CR geometry}. arXiv:2205.11022.

 \bibitem{Montano} Montano, B.C. (2005). \emph{Bi-Legendrian Connections}. Annales Polonici Mathematici, Vol 86, 79--95. 

\bibitem{Tanaka1} Tanaka, N. (1962). \emph{On the pseudo-conformal geometry of hypersurfaces of the space of $n$ complex variables}. Journal of the Mathematical Society of Japan, Vol 14(4), 397--429.

 \bibitem{Tanaka2} Tanaka, N. (1976). \emph{A differential geometric study on strongly pseudo-convex manifolds}. Lectures in Mathematics, Department of Mathematics, Kyoto University, No 9.

 \bibitem{Webster} Webster, S. (1978). \emph{Pseudo-Hermitian structures on a real hypersurface}. Journal of Differential Geometry, Vol 13(1), 25--41.

 \bibitem{Zalabova} Zalabov{\'a}, L. (2010). \emph{Symmetries of parabolic contact structures}. Journal of Geometry and Physics, Vol 60(11), 1698--1709.
\end{thebibliography}
\end{document}